\documentclass[a4paper, 10pt]{amsart}
\usepackage{amsmath, amscd}
\usepackage{amssymb, color}
\usepackage[all]{xy}
\usepackage{hyperref}
\usepackage{url}
\usepackage{enumitem}
\usepackage{latexsym}
\usepackage{mathrsfs}

\hypersetup{
    pdfmenubar=false,
    pdffitwindow=true,
    pdfstartview=FitH,
    colorlinks=true,
    linkcolor=red,
    citecolor=green,
    urlcolor=cyan
}

\theoremstyle{plain}
\newtheorem{theorem}{\bf Theorem}[section]
\newtheorem{proposition}[theorem]{\bf Proposition}
\newtheorem{lemma}[theorem]{\bf Lemma}
\newtheorem{corollary}[theorem]{\bf Corollary}

\theoremstyle{definition}
\newtheorem{example}[theorem]{\bf Example}
\newtheorem{examples}[theorem]{\bf Examples}

\newtheorem{remark}[theorem]{\bf Remark}

\newcommand{\N}{\mathbb N}
\newcommand{\Z}{\mathbb Z}
\newcommand{\R}{\mathbb R}
\newcommand{\Q}{\mathbb Q}

\DeclareMathOperator{\coker}{Coker}
\DeclareMathOperator{\spec}{spec} \DeclareMathOperator{\supp}{supp}
\DeclareMathOperator{\Pic}{Pic} 
\DeclareMathOperator{\id}{id}
\DeclareMathOperator{\Ker}{Ker}

\renewcommand{\time}{\negthinspace \times \negthinspace}

\newcommand{\DP}{\negthinspace : \negthinspace}

\newcommand{\red}{\text{\rm red}}
\newcommand{\eq}{\text{\rm eq}}
\newcommand{\adj}{\text{\rm adj}}
\newcommand{\monn}{\text{\rm mon}}

\newcommand{\FF}{\text{\rm FF}}
\newcommand{\mmod}{\negthickspace \mod}

\renewcommand{\t}{\, | \,}

\newcommand{\cls}[2]{\mathcal{C}(#1,#2)}
\newcommand{\quot}{\mathsf{q}}

\newcommand{\cat}{\mathsf{c}}
\newcommand{\ve}{\varepsilon}
\newcommand{\ZZ}{\mathsf{Z}}
\newcommand{\dd}{\mathsf{d}}
\newcommand{\wh}{\widehat}

\numberwithin{equation}{section}

\begin{document}

\title{Arithmetic of seminormal weakly Krull monoids and domains}

\address{Institut f\"ur Mathematik und Wissenschaftliches Rechnen\\
Karl--Fran\-zens--Universit\"at Graz, NAWI Graz\\
Heinrichstra{\ss}e 36\\
8010 Graz, Austria}

\email{alfred.geroldinger@uni-graz.at, florian.kainrath@uni-graz.at, andreas.reinhart@uni-graz.at}

\author{Alfred Geroldinger and Florian Kainrath and Andreas Reinhart}

\thanks{This work was supported by
the Austrian Science Fund FWF (Project Number P26036-N26).
\\}



\begin{abstract}
We study the arithmetic of seminormal $v$-noetherian weakly Krull monoids with nontrivial conductor which have finite class group and prime divisors in all classes. These monoids include seminormal  orders in holomorphy rings in global fields. The crucial property of seminormality allows us to give precise arithmetical results analogous to the well-known results for Krull monoids having finite class group and prime divisors in each class. This allows us to show, for example, that unions of sets of lengths are  intervals and to provide a characterization of half-factoriality.
\end{abstract}

\maketitle

\bigskip
\section{Introduction}
\bigskip

Let $R$ be a noetherian domain. Then every nonzero nonunit $a \in R$ can be written as a finite product of atoms (irreducible elements), say $a = u_1 \cdot \ldots \cdot u_k$. Such a product is called a factorization of $a$ in $R$. The main goal of factorization theory is to describe the various phenomena of non-uniqueness of factorizations by suitable arithmetical invariants (such as sets of lengths and unions of sets of lengths) and to study their relationship with classical ring-theoretical parameters of the underlying domain. Given an element $a \in R$, the set $\mathsf L (a)$ of all possible factorization lengths $k \in \N$ is called the set of lengths of $a$. Since $R$ is noetherian, $\mathsf L(a)$ is a finite subset of positive integers.  For $k \in \N$, let $\mathcal U_k (R)$ denote the set of all $l \in \N$ with the following property: There are atoms $u_i$ and $v_j$ for indices $i \in [1,k]$ and $j \in [1,l]$ such that $u_1 \cdot \ldots \cdot u_k = v_1 \cdot \ldots \cdot v_l$. Thus $\mathcal U_k (R)$ is the union of all sets of lengths containing $k$.
In particular, a domain $R$ is said to be half-factorial if $|\mathsf L (a)|=1$ for every nonzero nonunit $a \in R$ (equivalently, $\mathcal U_k (R) = \{k \}$ for all $k \in \N$). We note that if $R$ is not half-factorial, then there exists $a \in R$ with $|\mathsf L (a)| > 1$ and hence $|\mathsf L (a^N)| > N$ for each positive integer $N \in \N$. Throughout this manuscript we will study, for a seminormal weakly Krull domain $R$, sets of lengths $\mathsf L(a)$ where $a \in R$ and unions of sets of lengths $\mathcal U_k(R)$ where $k \in \mathbb N$.

\smallskip
Within factorization theory, there are two main cases. The first --- and by far the best understood --- case is that of Krull domains.  Recall that a noetherian domain is Krull if and only if it is integrally closed. Let $R$ be a Krull domain. Then arithmetical phenomena depend only on the class group and on the distribution of prime divisors in the classes.  In particular, $R$ is factorial if and only if the class group is trivial. Let $G$ denote the class group of $R$ and let $G_P \subset G$ denote the set of classes containing prime divisors. Then there is a transfer homomorphism from $R$ to the monoid $\mathcal B (G_P)$ of zero-sum sequences over $G_P$, which preserves sets of lengths and other invariants (see Section \ref{4}). This is the basis for a full variety of arithmetical finiteness results for Krull domains. If the class group is finite and every class contains a prime divisor, even more is known. Suppose this is the case; that is, $|G|<\infty$ and $G_P = G$. Then the natural transfer homomorphism provides  a perfect link to additive group and number theory. In particular, from this transfer homomorphism one is able to  establish not only finiteness results, but precise arithmetical results (see \cite[Chapter 6]{Ge-HK06a} and \cite{Ge09a} for an overview). We now demonstrate this by way of a few examples. The first such result --- indeed one of the first results in this area of factorization theory, due to Carlitz 1960 --- is a characterization of half-factoriality: the domain $R$ is half-factorial if and only if $|G| \le 2$. Secondly, it is known that  the unions of sets of lengths $\mathcal U_k(R)$ are not only finite  (which is true even in many non-Krull settings), but they are finite intervals. Precise values of other arithmetical invariants such as elasticity and the catenary degree have been determined using only  the structure of the class group $G$.

\smallskip
Much less is known in the non-Krull case. Suppose $R$ is noetherian but not integrally closed. Apart from certain classes of semigroup rings, the most investigated class is that of C-domains. Let $\overline R$ denote the integral closure and let $\mathfrak f = (R \DP \overline R)$ denote the conductor. If $\overline R$ is a finitely generated $R$-module and both $R/\mathfrak f$ and the class group $\mathcal C (\overline R)$ are finite, then $R$ is  a C-domain (for more examples see \cite{Re13a}). Here, a variety of finiteness results are known (see \cite{Ge-HK06a, Fo-Ge05, Fo-Ha06a}, and \cite{Ge-Ha08a, Ka14a} for generalizations).  However, even for the most simple C-domains (which are not Krull) there are more or less no precise arithmetical results. To give a striking example, consider a one-dimensional noetherian domain with finite Picard group and suppose that every class contains a prime ideal. For example, every non-principal order in an algebraic number field is such a domain. With no lack of effort, there has been no characterization of half-factoriality and no characterization of the structure of unions of sets of lengths for such a domain.

\smallskip
The study of seminormal commutative rings goes back to the work of Traverso and Swan (\cite{Tr70a, Sw80a}).
We refer to the survey of Vitulli \cite{Vi10a} and to the monographs \cite{Br-Gu09a, We13a}  for their role in  algebraic geometry and K-Theory.
The  theory of divisorial ideals of  Mori domains and monoids  was a central topic in multiplicative ideal theory over the past several  decades (see the classic text by Gilmer \cite{Gi92}, a survey by Barucci \cite{Ba00}, and  more recent  monographs  and proceedings \cite{HK98, Ge-HK06a, F-K-O-S11a}). Over the years, it has become evident that the seminormality condition  plays an important role in answering many of the posed algebraic questions (see for example \cite{Ba94}).

\smallskip
However, thus far the impact of the seminormality condition on the arithmetic of $R$ has not been studied, and this is the goal of the present paper.
We study the arithmetic of seminormal Mori domains and monoids which are  weakly Krull. These domains generalize Krull domains and include all seminormal one-dimensional noetherian domains (we will repeat the concept of weakly Krull domains at the beginning of Section \ref{5}, and provide a detailed discussion of examples in \ref{5.7}). We  derive several precise arithmetical results analogous to those known for Krull monoids with finite class group where each class contains a prime divisor, and demonstrate by examples that these results do not hold without the seminormality assumption.

\smallskip
In Section \ref{2} we gather together the arithmetical concepts  needed in the sequel. As is usual in factorization theory and in parts of multiplicative ideal theory, we proceed in  the setting of monoids. Apart from being more general, this procedure exhibits the purely multiplicative character of the theory. For example, a domain is a weakly Krull domain if and only if its multiplicative monoid is a weakly Krull monoid. Moreover, our main tools are those of transfer homomorphisms which allow us to shift problems from the original weakly Krull monoids to a simpler class of weakly Krull monoids which are easier to study. We note that even if one begins with a domain $R$, the transfer homomorphism is to a monoid and not to a domain; see Section \ref{4} and Lemma \ref{5.2}. In Section \ref{3} we study the local case as a preparation for the general (global) case of seminormal weakly Krull monoids which is handled in Section \ref{5}. Throughout, we develop in parallel the algebraic and the arithmetic properties. Our main results are stated in Theorem \ref{5.5} (algebraic structure), Theorem \ref{5.6} (arithmetic structure), and Theorem \ref{6.2} (characterization of half-factoriality). The relevance of the arithmetical results is discussed just before the formulation of Theorem \ref{5.6}. In Example \ref{5.12} we will show that    each of  the   statements in Theorem \ref{5.6} fails if one weakens some of the assumptions slightly, in particular if we drop the seminormality condition. Finally, we provide a new and state of the art outline for the study of half-factoriality at the beginning of Section \ref{6}.

\medskip
\section{Arithmetical preliminaries} \label{2}
\medskip

We denote by $\N$ the set of positive integers, and we put $\N_0 =
\N \cup \{0\}$.  For real numbers $a, b \in \R$, we
set $[a, b] = \{ x \in \Z \mid a \le x \le b \}$. Let $L, L' \subset
\Z$.  We denote by $L+L' = \{a+b \mid a \in L,\, b \in L' \}$ their
{\it sumset}. Two distinct elements $k, l \in L$ are called {\it
adjacent} if $L \cap [\min\{k,l\},\max\{k,l\}]=\{k,l\}$. A positive
integer $d \in \N$ is called a \ {\it distance} \ of $L$ \ if there
exist adjacent elements $k,l \in L$ with $d=|k-l|$, and we denote by $\Delta (L)$ the {\it set of  distances}.

In this preliminary section we gather the required  terminology from factorization theory. Our main reference is \cite{Ge-HK06a}.

\smallskip
\noindent {\bf Monoids and factorizations}. By a {\it monoid}, we
mean a commutative, cancelative semigroup with unit element.  Let
$H$ be a monoid. We denote by $\mathcal A (H)$ the set of atoms
(irreducible elements) of $H$, by $H^{\times}$ the group of
invertible elements, by $H_{\red} = H/H^{\times}= \{ a H^{\times} \mid a \in H \}$
the associated reduced monoid of $H$, and by $\mathsf q (H)$ the
quotient group of $H$. Each monoid homomorphism  $\varphi \colon H \to D$  induces a monoid homomorphism $\varphi_{\red} \colon H_{\red} \to D_{\red}$, defined by $\varphi (aH^{\times}) = a D^{\times}$ for all $a \in H$, and a group homomorphism $\mathsf q ( \varphi) \colon \mathsf q (H) \to \mathsf q (D)$  satisfying $\mathsf q (\varphi) \t H = \varphi$. If $D$ is a monoid and $H \subset D$ a submonoid, then $H \subset D$ is said to be {\it saturated} if $H = \mathsf q (H) \cap D$.
For a domain $R$, we denote by $R^{\bullet}$ its multiplicative semigroup of nonzero elements, and obviously this is a monoid. All arithmetical terms defined for monoids carry over to domains, and we set $\mathcal A (R) = \mathcal A (R^{\bullet})$, and so on.

For a set $P$, we denote by $\mathcal F (P)$ the \ {\it free
abelian monoid} \ with basis $P$. Then every $a \in \mathcal F
(P)$ has a unique representation in the form
\[
a = \prod_{p \in P} p^{\mathsf v_p(a) } \quad \text{with} \quad
\mathsf v_p(a) \in \N_0 \ \text{ and } \ \mathsf v_p(a) = 0 \ \text{
for almost all } \ p \in P \,.
\]
We call $|a|= \sum_{p \in P}\mathsf v_p(a)$ the \emph{length} of $a$
and $\supp (a) = \{p \in P \mid \mathsf v_p (a) > 0\} \subset P$ the
{\it support} of $a$.
The  monoid  $\mathsf Z (H) = \mathcal F \bigl(
\mathcal A(H_\red)\bigr)$  is called the  {\it factorization
monoid}  of $H$, and  the unique homomorphism
\[
\pi \colon \mathsf Z (H) \to H_{\red} \quad \text{satisfying} \quad
\pi (u) = u \quad \text{for each} \quad u \in \mathcal A(H_\red)
\]
is  the  {\it factorization homomorphism}  of $H$. For $a
\in H$ and $k \in \N$,
\[
\begin{aligned}
\mathsf Z_H (a) = \mathsf Z (a)  & = \pi^{-1} (aH^\times) \subset
\mathsf Z (H) \quad
\text{is the \ {\it set of factorizations} \ of \ $a$} \,, \\
\mathsf Z_{H,k} (a) = \mathsf Z_k (a) & = \{ z \in \mathsf Z (a) \mid |z| = k \} \quad
\text{is the \ {\it set of factorizations} \ of \ $a$ of length \
$k$}, \quad \text{and}
\\
\mathsf L_H (a) = \mathsf L (a) & = \bigl\{ |z| \, \bigm| \, z \in
\mathsf Z (a) \bigr\} \subset \N_0 \quad \text{is the \ {\it set of
lengths} \ of $a$}  \,.
\end{aligned}
\]
We denote by $\mathcal L (H) = \{ \mathsf L (a) \mid a \in H \}$ the {\it system of sets of lengths} of $H$.
The monoid $H$ is called
\begin{itemize}
\item {\it atomic} if $\mathsf Z (a) \ne \emptyset$ for all $a \in H$,
\item an {\it \FF-monoid} if $\mathsf Z (a)$ is finite and non-empty for all $a \in H$,
\item {\it factorial} if $|\mathsf Z (a)| = 1$ for all $a \in H$, and
\item {\it half-factorial} if $|\mathsf L (a)| = 1$ for all $a \in H$.
\end{itemize}
If $H$ is atomic but not half-factorial, then there is an $a \in H$ with $|\mathsf L (a)| > 1$, and a simple calculation shows that $|\mathsf L (a^N)|>N$   for each $N \in \N$.
In a $v$-noetherian monoid all sets of lengths are finite. We discuss some parameters describing the non-uniqueness of factorizations.
Suppose that $H$ is atomic.

\medskip
\noindent
{\bf Lengths of factorizations}.
Sets of lengths and invariants derived from them are the best investigated parameters in factorization theory.
Let  $k \in \N$. If $H = H^{\times}$, then we set $\mathcal U_k (H) = \{k\}$. If $H \ne H^{\times}$, then
\[
\mathcal U_k (H) = \bigcup_{k \in L, L \in \mathcal L (H)} L
\]
is the union of all sets of lengths containing $k$. In both cases, we define
\[
\rho_k (H) = \sup \mathcal U_k (H) \quad \text{and} \quad \lambda_k (H) = \min \mathcal U_k (H) \,.
\]
Next, let
\[
\Delta (H) = \bigcup_{L \in \mathcal L (H)} \Delta (  L )
\]
denote the {\it set of distances} of $H$. By definition, $H$ is  half-factorial if and only  if $\Delta (H) = \emptyset$ (equivalently, $\mathcal U_k (H) = \{k\}$ for all $k \in \N$). We will need the following simple lemma.

\smallskip
\begin{lemma} \label{2.1}
Let $H$ be an atomic monoid and $n \in \N$. If $\mathcal U_n (H) \supset \N_{\ge n}$, then $\mathcal U_k (H) \supset \N_{\ge n}$ for all $k \ge n$.
\end{lemma}

\begin{proof}
We proceed by induction on $k$. We suppose that $n \in \N$, \ $k \ge n$ and \,$\mathcal U_l(H) \supset \N_{\ge n}$ for all $l \in [n,k]$, and we show that $\mathcal U_{k+1}(H) \supset \N_{\ge n}$. If $m \in [n,k]$, then $k+1 \in \mathcal U_m(H)$ by assumption, and thus also $m \in \mathcal U_{k+1}(H)$. If $m \ge k+1> n$, then $m-1 \in \mathcal U_k(H)$, and therefore there exist atoms $u_1, \ldots, u_k,\,v_1, \ldots, v_{m-1}$ such that $u_1 \cdot \ldots \cdot u_k = v_1 \cdot \ldots \cdot v_{m-1}$. Hence $u_1 \cdot \ldots \cdot u_k w = v_1 \cdot \ldots \cdot v_{m-1}w$ for any atom $w \in H$, and therefore $m \in \mathcal U_{k+1}(H)$.
\end{proof}

\medskip
\noindent
{\bf Distances of factorizations}.
Let $z,\, z' \in \mathsf Z (H)$. Then we can write
\[
z = u_1 \cdot \ldots \cdot u_lv_1 \cdot \ldots \cdot v_m \quad
\text{and} \quad z' = u_1 \cdot \ldots \cdot u_lw_1 \cdot \ldots
\cdot w_n\,,
\]
where  $l,\,m,\, n\in \N_0$ and $u_1, \ldots, u_l,\,v_1,
\ldots,v_m,\, w_1, \ldots, w_n \in \mathcal A(H_\red)$ are such that
\[
\{v_1 ,\ldots, v_m \} \cap \{w_1, \ldots, w_n \} = \emptyset\,.
\]
Then $\gcd(z,z')=u_1\cdot\ldots\cdot u_l$, and we call
\[
\dd (z, z') = \max \{m,\, n\} = \max \{ |z \gcd (z, z')^{-1}|,
|z' \gcd (z, z')^{-1}| \} \in \N_0
\]
the {\it distance} between $z$ and $z'$. If $\pi (z) = \pi (z')$ and
$z \ne z'$, then
\begin{equation} \label{basic1}
2 + \bigl| |z |-|z'| \bigr| \le \dd (z, z')
\end{equation}
by \cite[Lemma 1.6.2]{Ge-HK06a}. For subsets $X, Y \subset \mathsf Z
(H)$, we set
\[
\dd (X, Y) = \min \{ \dd (x, y ) \mid x \in X, \, y \in
Y \} \,,
\]
and thus  $\dd (X, Y) = 0$  if and only if ( $X  \cap Y \ne
\emptyset$ or $X= \emptyset$ or $Y = \emptyset$ ).

\medskip
\noindent
{\bf Chains of factorizations}. Let $a \in H$ and
$N \in \N_0 \cup \{\infty\}$. A finite sequence $z_0, \ldots,
z_k \in \mathsf Z (a)$ is called a {\it $($monotone$)$ $N$-chain of
factorizations} if $\dd (z_{i-1}, z_i) \le N$ for all $i \in
[1, k]$ (and $|z_0| \le \ldots \le |z_k|$ or $|z_0| \ge \ldots \ge
|z_k|$). We denote by  $\mathsf c (a) \in \N _0 \cup
      \{\infty\}$ (or by $\mathsf c_{\monn} (a)$ resp.)  the smallest $N \in \N _0 \cup \{\infty\}$ such
      that any two factorizations $z,\, z' \in \mathsf Z (a)$ can be
      concatenated by an $N$-chain (or by a monotone $N$-chain resp.).
Then
\[
\mathsf c(H) = \sup \{ \mathsf c(b) \mid b \in H\} \in \N_0 \cup \{\infty\} \quad \text{and} \quad
\mathsf c_{\monn} (H) = \sup \{ \mathsf c_{\monn} (b) \mid b \in H\} \in \N_0 \cup \{\infty\} \quad
\,
\]
denote  the \ {\it catenary degree} \ and the \ {\it monotone
catenary degree} of $H$. The monotone catenary degree is studied by
using the two auxiliary notions of the equal and the adjacent
catenary degrees. Let $\mathsf c_{\eq} (a)$ denote the smallest $N
\in \N_0 \cup \{\infty\}$ such that any two factorizations $z, z' \in \mathsf Z (a)$ with $|z| = |z'|$ can be concatenated by a monotone $N$-chain.
We call
      \[
      \mathsf c_{\eq} (H) = \sup \{ \mathsf c_{\eq} (b) \mid b \in H \} \in \N_0 \cup \{\infty\}
      \]
      the {\it equal catenary degree} of $H$.
We set
      \[
      \mathsf c_{\adj}(a)  = \sup \{ \dd \big( \mathsf Z_k (a),
      \mathsf Z_l (a) \big) \mid k, l \in \mathsf L (a) \ \text{are
      adjacent} \} \,,
      \]
      and the {\it adjacent catenary degree} of $H$ is defined as
      \[
      \mathsf c_{\adj} (H) = \sup \{ \mathsf c_{\adj} (b) \mid b \in H \} \in \N_0 \cup
      \{\infty\} \,.
      \]
Obviously,  we have
\[
\mathsf c (a) \le \mathsf c_{\monn} (a) = \sup \{ \mathsf c_{\eq} (a), \mathsf c_{\adj} (a) \} \le
\sup \mathsf L (a) \quad \text{for all} \quad a \in H \,,
\]
and hence
\begin{equation}
\mathsf c (H) \le \mathsf c_{\monn} (H) = \sup \{ \mathsf c_{\eq} (H), \mathsf c_{\adj} (H) \} \,.
\label{basic2}
\end{equation}
By definition, $H$ is
factorial if and only if $\mathsf c (H) = 0$ if and only if $\mathsf c_{\monn} (H) = 0$. If $H$ is
not factorial, then Equation (\ref{basic1}) shows that $\mathsf c (H) \ge 2$. Again by definition, we have
$\mathsf c_{\adj} (H) = 0$ if and only if $H$ is half-factorial, and if this holds, then $\mathsf c
(H) = \mathsf c_{\eq} (H) = \mathsf c_{\monn} (H)$. Furthermore, we will use the basic inequality
\begin{equation}
2 + \sup \Delta (H) \le \mathsf c (H) \,. \label{basic3}
\end{equation}
Therefore, $\mathsf c (H) = 2$ implies that $H$ is half-factorial, and $\mathsf c (H) = 3$ yields that $\Delta (H) \subset \{1\}$. Moreover,  $\mathsf c_{\eq} (H) = 0$ if and only if for all $a \in H$ and all $k \in \mathsf L (a)$ we have $|\mathsf Z_k (a)| = 1$. Thus a
recent result of Coykendall and Smith implies that a domain $R$ is factorial if and only if $\mathsf c_{\eq} (R) = 0$ (\cite[Corollary 2.12]{Co-Sm11a}; the analogous result is not true for monoids as it is outlined in the same paper).

\bigskip
\section{The local case: seminormal  finitely primary monoids} \label{3}
\bigskip

Weakly Krull monoids are intersections of finite character of their localizations at minimal primes
(see Section \ref{5}), and these localizations are primary.
So in order to understand the arithmetic of (seminormal) weakly Krull monoids we have to study the arithmetic of (seminormal) primary monoids first.
This will be done in the present section. The main result here is Theorem \ref{3.7}.

Let $D$ be a monoid. An element $q \in D$ is called {\it primary} if $q \notin D^{\times}$ and, for all $a, b \in D$,
$q \t ab$ implies $q \t a$ or $q \t b^n$ for some $n \in \N$. The monoid $D$ is called {\it primary} if one of the following equivalent conditions is satisfied:
\begin{itemize}
\item $D \ne D^{\times}$ and if every non-unit is primary,

\item $D \ne D^{\times}$ and $s$-$\spec (D) = \{\emptyset, D \setminus D^{\times} \}$),

\item $D \setminus D^{\times}$ and $D$ are the only non-empty radical ideal of $D$.
\end{itemize}
It is easy to verify that a saturated submonoid of a primary monoid is primary again.
We will mainly be concerned with a special class of primary monoids, called finitely primary.
The monoid $D$ is
called  {\it finitely primary} if there exist $s,\, \alpha \in \N$
such that $D$ is a submonoid of a factorial monoid $F= F^\times
\time [q_1,\ldots,q_s]$ with $s$ pairwise non-associated prime
elements $q_1, \ldots, q_s$ satisfying
\begin{equation} \label{finitelyprimary}
D \setminus D^\times \subset q_1 \cdot \ldots \cdot q_sF \quad
\text{and} \quad (q_1 \cdot \ldots \cdot q_s)^\alpha F \subset D
\,.
\end{equation}
Finitely primary
monoids serve as multiplicative models in the study of
one-dimensional local domains (see Lemma \ref{3.3}).
We start with a basic and well-known lemma. We provide its simple proof in order to show how factorization works in finitely primary monoids (arguments of this type will be used a lot in the present  paper).

\medskip
\begin{lemma} \label{3.0}
Let $D$ be a finitely primary monoid of rank $s \ge 2$ with all conventions  as in $(\ref{finitelyprimary})$. Let $a = \epsilon q_1^{k_1} \cdot \ldots \cdot q_s^{k_s} \in D$ for some $\epsilon \in F^{\times}$ and $k_1, \ldots, k_s \in \N$. Then $\min \mathsf L (a) \le 2 \alpha$ and $\max \mathsf L (a) \le \min \{k_1, \ldots, k_s\}$.
\end{lemma}

\begin{proof}
If $a=a_1 \cdot \ldots \cdot a_l$ with $a_1, \ldots, a_l \in D \setminus D^{\times}$, then $D \setminus D^\times \subset q_1 \cdot \ldots \cdot q_sF$ implies that $l \le \min \{k_1, \ldots, k_s\}$ and hence $\max \mathsf L (a) \le \min \{k_1, \ldots, k_s\}$. If $\max \mathsf L (a) > 2 \alpha$, then $a=a_1a_2$, where
\[
a_1 = \epsilon q_1^{\alpha} q_2^{k_2-\alpha} \cdot \ldots \cdot q_s^{k_s-\alpha} \quad \text{and} \quad a_2=q_1^{k_1-\alpha}q_2^{\alpha} \cdot \ldots \cdot q_s^{\alpha} \,,
\]
and hence $\min \mathsf L (a) \le \min \mathsf L (a_1)+\min \mathsf L (a_2) \le \max \mathsf L (a_1)+\max \mathsf L (a_2) \le \alpha + \alpha$.
\end{proof}

Our notation on ideal theory  follows \cite{HK98} with the obvious modifications induced by the fact that the monoids in this paper do not contain a zero-element. In particular, for subsets $A, B \subset \mathsf q (D)$, we denote by
$(A \DP B) = \{ x \in \mathsf q (D) \mid x B \subset A \}$, by $A^{-1} = (D \DP A)$, and by $A_v = (A^{-1})^{-1}$.  By an ideal of $D$ we always mean an $s$-ideal, and an $s$-ideal $A$ is a $v$-ideal if $A_v = A$. We denote by $\mathcal F_v (D)$ the set of all fractional $v$-ideals and by $\mathcal I_v (D)$ the set of all $v$-ideals of $D$. Furthermore, $\mathcal I_v^* (D)$ is the monoid of $v$-invertible $v$-ideals (with $v$-multiplication) and $\mathcal F_v (D)^{\times} = \mathsf q \big( \mathcal I_v^* (D) \big)$ is its quotient group of fractional invertible $v$-ideals.
By  $\mathfrak X (D)$, we denote the set of all minimal non-empty prime $s$-ideals of $D$. Note that the localization $D_{\mathfrak p}$ is primary for each $\mathfrak p \in \mathfrak X (D)$.
We denote by
\begin{itemize}
\item $D' = \{ x \in \mathsf q (D) \mid \ \text{there exists some} \ N \in
\N \
             \text{such that} \ x^n \in D \ \text{for all} \ n \ge N \}
$ the {\it seminormal closure} (also called the {\it
seminormalization}) of $D$, and by

\smallskip
\item $\wh D = \{ x \in \mathsf q (D) \mid \ \text{there exists
             some } \ c \in D \ \text{such that} \ cx^n \in D \ \text{for all}
             \ n \in \N \}$ the {\it complete integral closure} of
             $D$ \,,
\end{itemize}
and observe that $
D \subset D'  \subset \wh D \subset \mathsf
q (D)$. The monoid $D$ is called
\begin{itemize}
\item {\it completely integrally closed} if $D = \wh D$, and

\item {\it seminormal} if $D = D'$ (equivalently,  if $x \in \mathsf q (D)$ and $x^2, x^3 \in D$, then
           $x \in D$).
\end{itemize}
It is easily checked that the  seminormalization $D'$ of  $D$ is seminormal (i.e., $(D')' = D')$,  that $D$ is primary if and only if $D'$ is primary, and that $D^{\times} = {D'}^{\times} \cap D$ (\cite[Propositions 1 and 2]{Ge96}).
A domain $R$ is seminormal if its multiplicative monoid $R^{\bullet}$ is seminormal (equivalently, the canonical monomorphism $\Pic (R) \to \Pic (R[X])$ is an isomorphism; see \cite[Section 2.4]{Vi10a} for a survey on this).

A monoid $D$ is said to be {\it archimedean} if $\wh D^{\times} \cap D = D^{\times}$ (equivalently, $\bigcap_{n \in \N} a^nD = \emptyset$ for all $a \in D \setminus D^{\times}$).  Every primary monoid, every $v$-noetherian monoid, every completely integrally closed monoid, and every domain satisfying Krull's Principal Ideal Theorem is archimedean. A monoid $D$ is called a {\it discrete valuation monoid} if $D_{\red} \cong (\N_0, +)$ (equivalently, $D$ is primary and contains a prime element).

If $D$ is finitely primary  and  all notation is as in Equation (\ref{finitelyprimary}), then obviously $F = \wh D$ and $s = |\mathfrak X (\wh D)|$. We  will use without further mention that $D$ is finitely primary if and only if one of  the
following equivalent conditions holds (see
\cite[Theorem 2.9.2]{Ge-HK06a}):
\begin{enumerate}
\item[(a)] $D$ is finitely primary of rank $s$.

\item[(b)] $D_{\red}$ is finitely primary of rank $s$.

\item[(c)] $D$ is primary, $(D \DP \wh D) \ne \emptyset$ and $\wh D _{\red} \cong \N_0^s$.
\end{enumerate}

\medskip
\begin{lemma} \label{3.1}
Let $D$ be a monoid.
\begin{enumerate}
\item If $S \subset D$ is a submonoid, then $(S^{-1}D)' = S^{-1}D'$.
      In particular, if $D$ is seminormal, then $S^{-1}D$ is seminormal.

\item If $(D_i)_{i \in I}$ is a family of monoids such that $D = \coprod_{i \in I} D_i$, then $D' = \coprod_{i \in I} D_i'$. In particular, $D$ is seminormal
      if and only if $D_i$ is seminormal for all $i \in I$.

\item $D$ is seminormal if and only if $D_{\red}$ is seminormal.

\item If $D$ is seminormal and   $H \subset D$ is a saturated submonoid, then $H$ is seminormal.
\end{enumerate}
\end{lemma}

\begin{proof}
1. is well-known (see, for example, \cite[Lemma 2.4]{Re13a}). 2. and 3. are easy to check. To show 4.,
let $x \in \mathsf q (H)$ such that $x^2, x^3 \in H$. Then $x \in \mathsf q (D)$, $x^2, x^3 \in D$, hence $x \in D$, and therefore $x \in \mathsf q (H) \cap D = H$.
\end{proof}

\medskip
\begin{lemma} \label{3.2}
Let $D$ be a seminormal monoid and $D \subsetneq S \subset \mathsf q(D)$ an overmonoid.
\begin{enumerate}
\item
$(D \DP S)$ is a radical ideal both of $D$ and of $S$. If \,$D$ is primary and $S \subset \wh D$, then $(D \DP S) = (D \DP \wh D) = D \setminus D^\times$.

\smallskip

\item
If \,$S$ is primary and $(D \DP S) \ne \emptyset$, then $S$ is seminormal.

\smallskip

\item
Let $D$ be archimedean, $(D \DP \wh D) \ne \emptyset$, and let $\wh D$ be primary. Then $\wh D \setminus \wh D^\times = D \setminus D^\times$, and if \,$\wh D$ is $v$-noetherian, then $D$ is $v$-noetherian, too.

\smallskip

\item
Suppose that $D$ is primary, $S$ is archimedean, $S \subset \wh D$, and let $\wh D$ be $v$-noetherian and primary. Then $S$ is seminormal, $v$-noetherian and primary.
\end{enumerate}
\end{lemma}

\begin{proof}
1. Obviously, $(D \DP S)$ is an ideal both of $S$ and of $D$. If $x \in \sqrt{(D\DP S)}$, then $x^kS \subset D$ for some $k \in \N$, and therefore $(xy)^l \in D$ for all $y \in S$ and all $l \ge k$. Since $D$ is seminormal, it follows that $xy \in D$ for all $y \in S$ and therefore $x \in (D \DP S)$. Hence $(D \DP S) = \sqrt{(D \DP S)}$ is a radical ideal.

Let now $D$ be primary. Then $D \supsetneq (D \DP S) \supset (D \DP \wh D) \ne \emptyset$ by \cite[Proposition 4.8]{G-HK-H-K03} (observe that every primary monoid is a G-monoid), and thus $(D \DP S)=(D \DP \wh D) = D \setminus D^\times$, since this is the only non-empty radical ideal of $D$ distinct from $D$.

\smallskip

2. By 1., $(D \DP S)$ is a radical ideal of $S$, hence $(D \DP S) = S \setminus S^\times \subset D$. Suppose now that $x \in \mathsf q(S)= \mathsf q(D)$ and $x^2, \, x^3 \in S$. If $x^2\in S^{\times}$ or $x^3\in S^{\times}$, then $(x^2)^3 = (x^3)^2 \in S^{\times}$, hence $x^2, x^3 \in S^{\times}$ and $x = x^{-2}x^3 \in S$. Otherwise, $x^2 \notin S^{\times}$ and $x^3 \notin S^{\times}$, and hence  $x^2, x^3 \in D$ and  $x \in D \subset S$.

\smallskip

3. By 1., $(D \DP \wh D)$ is a radical ideal of $\wh D$, hence $(D \DP \wh D) = \wh D \setminus \wh D^\times \subset D$, and since $\wh D^\times \cap D = D^\times$, we obtain $\wh D \setminus \wh D^\times = D \cap (\wh D \setminus \wh D^\times) = D \setminus D^\times$. By \cite[Proposition 2.2]{Re13a}, a monoid $H$ is $v$-noetherian if and only if $H \setminus H^\times$ is a $\mathsf q(H)$-Mori set. If $\wh D$ is $v$-noetherian, then  $\wh D \setminus \wh D^\times = D \setminus D^\times$ is a $\mathsf q(\wh D)$-Mori set, and since $\mathsf q(D) = \mathsf q(\wh D)$ it follows that $D$ is $v$-noetherian.

\smallskip

4. Since $\wh D = \wh S$, we get $D \setminus D^\times = (D \DP S) = ( D \DP \wh D) = (D \DP \wh S) \subset (S \DP \widehat S)$.  If $\emptyset \ne \mathfrak p \in s\text{-}\spec(S)$, then $\emptyset \ne \mathfrak p \cap D \in s\text{-}\spec(D)$ and
\[
\wh D \setminus \wh D^\times = D \setminus D^\times = \mathfrak p \cap D \subset \mathfrak p \subset S \setminus S^\times = S \setminus ( \widehat{S}^{\times} \cap S) = S \setminus \widehat{S}^{\times} \subset \widehat S \setminus \widehat{S}^{\times} = \widehat D \setminus \widehat{D}^{\times} \,.
\]
Hence $\mathfrak p = S \setminus S^\times$, and therefore $S$ is primary. Hence $S$ is seminormal by 2., and by 3. (applied with $S$ instead of $D$) we infer that $S$ is $v$-noetherian.
\end{proof}

\smallskip
We are going to address C-monoids from time to time.
We do not want to repeat (the rather involved) definition of C-monoids. However, since C-monoids (from C-monoids in general Krull domains to arithmetical congruence monoids in the integers; see also Examples \ref{5.7}.2) have found some attention in recent literature (\cite{Ge-HK06a, Re13a, Ge-Ra-Re14c}), we want to outline the connection with the present concepts. However, C-monoids play no role in our main results (Theorems \ref{5.6} and \ref{6.2}), and hence the reader can skip these comments.

\medskip
\begin{lemma} \label{3.3}~
\begin{enumerate}
\item Let $D \subset F = F^\times \time [q_1,\ldots,q_s]$ be a finitely
primary monoid of rank $s$ and exponent $\alpha$. Then
\[
D' = q_1 \cdot \ldots \cdot q_s F \cup {D'}^{\times} \,, \quad
 \quad \wh D = \wh{D'} = F \,,
\]
and  $D'$ is a seminormal    finitely primary monoid
of rank $s$ and exponent $1$. Moreover, $D'$ is a {\rm C}-monoid and $v$-noetherian.

\smallskip
\item A domain $R$ is one-dimensional and local if and only if $R^{\bullet}$ is primary.

\smallskip
\item For a domain $R$ the following statements are equivalent{\rm \,:}
      \begin{enumerate}
      \smallskip
      \item $R$ is a seminormal one-dimensional local Mori domain.

      \smallskip
      \item $R^{\bullet}$ is seminormal  finitely primary.
      \end{enumerate}
\end{enumerate}
\end{lemma}

\begin{proof}
1. Since  $D \subset D' \subset \wh D = F$, it follows that  $F = \wh D \subset  \wh{D'} \subset \wh F = F$. Suppose that $a = \varepsilon q_1^{k_1} \cdot \ldots \cdot q_s^{k_s} \in F$, where $\varepsilon \in F^\times$ and $k_1, \ldots, k_s \in \N_0$. If $a \in D'$, then $a^n \in D$ for all sufficiently large $n \in \N$, and therefore either $k_1= \ldots = k_s =0$ and $\varepsilon \in D^{\prime \times}$, or $k_1, \ldots, k_s \in \N$. Hence it follows that $D' \subset q_1 \cdot \ldots \cdot q_sF \cup D^{\prime \times}$, and as the opposite inclusion is obvious, we get $D' = q_1 \cdot \ldots \cdot q_sF \cup D^{\prime \times}$. In particular, it follows that $D'$ is finitely primary of rank $s$ and exponent $1$. Moreover, it is easily checked that $D'$ satisfies the conditions of \cite[Corollary 2.9.8]{Ge-HK06a} (with $V = F^\times$ and $\alpha =1$). Hence $D'$ is a ${\rm C}$-monoid (even a ${\rm C}_0$-monoid) and thus it is $v$-noetherian by \cite[Theorem 2.9.13]{Ge-HK06a}.

\smallskip
2. See \cite[Proposition 2.10.7]{Ge-HK06a}.

\smallskip
3. (a)\, $\Rightarrow$\, (b) \ Since $R$ is seminormal, 2. and Lemma \ref{3.2}.1 imply that the conductor $(R \DP \wh R) \ne \{0\}$. Now  \cite[Proposition 2.10.7]{Ge-HK06a} implies that $R^{\bullet}$ is finitely primary.

(b)\, $\Rightarrow$\, (a) \ By 1., $R$ is a Mori domain and by 2., $R$ is one-dimensional and local.
\end{proof}

\medskip
In contrast to the above result,  there are seminormal primary \FF-monoids which are not $v$-noetherian and hence not finitely primary (\cite[Example 4.7]{Ge-Ha08a}). Let $H$ be a  seminormal  $v$-noetherian primary monoid. If $H$ stems from a domain, then $H$ is finitely primary by Lemma \ref{3.3}. But this does not hold in general as the following example shows (thus the assumption -- that all $H_{\mathfrak p}$ are finitely primary --  made in the results \ref{5.5},  \ref{L:CharClsIso}, \ref{5.6}, and others is indeed an additional one).

\smallskip
\begin{example} \label{3.4}
Let $S$ be a non-factorial Krull monoid with divisor theory $S \hookrightarrow F = \mathcal F (P)$ where $P = \{q_1, \ldots, q_s\}$ for some $s \ge 2$. Then, by Lemma \ref{3.3}.1,
\[
D = q_1 \cdot \ldots \cdot q_s F \cup \{1\} \subset F
\]
is a seminormal $v$-noetherian finitely primary monoid with $\wh D = F$. Now we define $H = D \cap \mathsf q (S)$. Since the inclusion $S \hookrightarrow F$ is a divisor theory, there is an element $a \in S \cap D$. Then for every $b \in S$ we have
\[
b = \frac{ba}{a} \in \mathsf q (D \cap S) \subset \mathsf q (H) \subset \mathsf q (S) \,,
\]
hence $\mathsf q (S) = \mathsf q (H)$ and $H = D \cap \mathsf q (H)$. Thus $H \subset D$ is saturated and primary. Furthermore,
\[
\wh H = \wh D \cap \mathsf q (S) = F \cap \mathsf q (S) = S \,,
\]
which implies that  $H$ is seminormal and $v$-noetherian  with $(H \DP \wh H) \ne \emptyset$ by Lemma \ref{3.1}.4 and Lemma \ref{3.2}.1. But $H$ is not finitely primary since $\wh H = S$ is not factorial.
\end{example}

Following the terminology of \cite{A-A-H-K-M-P-R09, B-P-A-A-H-K-M-R11, Ba-Ch13a}, we say that
 a monoid $D$ is  {\it bifurcus} if it is atomic with $D \ne D^{\times}$ and $\min \mathsf L (a) = 2$ for all $a \in D \setminus ( D^{\times} \cup \mathcal A (D) )$. Obviously, $D$ is bifurcus if and only if $D_{\red}$ is bifurcus.

\medskip
\begin{lemma} \label{3.5}
Let $D \subset F = F^\times \time [q_1,\ldots,q_s]$ be a seminormal
finitely primary monoid of rank $s$ and exponent $\alpha$. Then
\begin{enumerate}
\item $\mathcal A (D) = \{ \epsilon q_1^{k_1} \cdot \ldots \cdot q_s^{k_s} \mid  \epsilon \in F^{\times} \  \text{\rm and} \
      \min \{k_1, \ldots, k_s \} = 1 \}$.

\smallskip
\item If $s=1$, then $D$ is not bifurcus, $\mathsf c (D) \le 2$, and $D$ is half-factorial.

\smallskip
\item If $s \ge 2$, then $D$ is bifurcus, and $\mathsf c_{\eq} (D)\leq 2$.

\smallskip
\item If $S$ is a primary monoid with $D \subset S \subset \wh D$, then $S$ is  seminormal finitely primary,   and $\mathcal A (S) = \mathcal A (D)$.
\end{enumerate}
\end{lemma}

\begin{proof}
1. Lemma \ref{3.3} shows that
\[
D = D' = q_1 \cdot \ldots \cdot q_s F \cup D^{\times} \,,
\]
which implies immediately that the set of atoms is as asserted.

For the proof of 2. and 3. we may assume that $D$ is reduced.

\smallskip

2. Set $q = q_1$, and suppose that $a = \varepsilon q^n \in D$, where $n \in \N$ and $\varepsilon \in D^\times$. Then $\mathsf L(a) = \{n\}$, hence $D$ is half-factorial and not bifurcus. Since $z^* = q^{n-1} (\varepsilon q) \in \mathsf Z(a)$, it suffices to prove that for every factorization $z \in \mathsf Z(a)$ there is a $2$-chain concatenating $z$ and $z^*$. If $z \in \mathsf Z(a)$, then $z = (\varepsilon_1 q)\cdot \ldots \cdot (\varepsilon_nq)$, where $\varepsilon_1, \ldots, \varepsilon_n \in D^\times$ and $\varepsilon = \varepsilon_1 \cdot \ldots \cdot \varepsilon_n$. For $i \in [1,n]$, we set
\,$z_i = q^{i-1}(\varepsilon_1 \cdot \ldots \cdot \varepsilon_iq)(\varepsilon_{i+1}q) \cdot \ldots \cdot (\varepsilon_nq) \in \mathsf Z(a)$. Then $z_1 = z$, \ $z_n = z^*$ and $\mathsf d(z_i,z_{i+1}) \le 2$ for all $i \in [1,n-1]$. Hence $\mathsf c(a) \le 2$, and thus $\mathsf c(D) \le 2$.

3. By Lemma \ref{3.0}, $D$ is bifurcus. Let $a  = \varepsilon q_1^{\alpha_1} \cdot \ldots \cdot q_s^{\alpha_s} \in D$ \,(where $\varepsilon \in F^\times$ and $\alpha_1, \ldots, \alpha_s \in \N$), \ $z \in \mathsf Z(a)$ and $|z| = k \ge 2$. Then it follows that $\alpha_i \ge k$ for all $i \in [1,s]$, and we define $v_1,\,v_2, \,v_3 \in \mathsf Z(D)$ by
\[
v_1 = q_1 \cdot \ldots \cdot q_s, \quad v_2 = q_1 \cdot \ldots \cdot q_{s-1} q_s^{\alpha_s -k+1} \ \text{ and } \ v_3 = \varepsilon q_1^{\alpha_1-k+1} \cdot \ldots \cdot q_{s-1}^{\alpha_{s-1}-k+1} q_s.
\]
Then \,$z^* = v_1^{k-2}v_2v_3 \in \mathsf Z(a)$, and we construct a $2$-chain of equal-length factorizations concatenating $z$ and $z^*$. We proceed in two steps.

In the first step the product of two atoms $w w'$, where
\[
w = \epsilon q_1^{k_1} \cdot \ldots \cdot q_s^{k_s} \quad \text{and} \quad
w' = \epsilon' {q_1}^{k_1'} \cdot \ldots \cdot q_s^{k_s'} \,,
\]
is replaced by the product of two atoms $v v'$, where
\[
v = q_1 \cdot \ldots \cdot q_{s-1}q_s^{k_s+k_s'-1} \quad \text{and} \quad
v' = \epsilon \epsilon' q_1^{k_1+k_1'-1} \cdot \ldots \cdot q_{s-1}^{k_{s-1}+k_{s-1}'-1}q_s \,.
\]
This process ends at a factorization
$z' = w_1 \cdot \ldots \cdot w_{k-1} v_3$, where $w_i = q_1 \cdot \ldots \cdot q_{s-1} q_s^{\beta_i}$ with some exponent $\beta_i \in \N$ for all $i \in [1,k-1]$, and $\beta_1 + \ldots + \beta_{k-1} = \alpha_s-1$.

In the second step the product of two atoms $w w'$, where
\[
w = q_1 \cdot \ldots \cdot q_{s-1} q_s^{k_s} \quad \text{and} \quad
w' = q_1 \cdot \ldots \cdot q_{s-1} q_s^{k_s'}
\]
is replaced by the product of two atoms $v_1 v'$, where
\[
v_1 = q_1 \cdot \ldots \cdot q_s \quad \text{and} \quad
v' = q_1 \cdot \ldots \cdot q_{s-1} q_s^{k_s + k_s' - 1} \,.
\]
This process ends at $z^*$.

\smallskip
4. Since $\wh D = F$ is factorial, we get $F = \wh D \subset \wh S \subset \wh F = F$ and hence $\wh S = F$. Since $S$ is primary and $\emptyset \ne (D \DP \wh D) \subset (S \DP \wh D) = (S \DP \wh S)$, it follows that $S$ is finitely primary. Since $\emptyset \ne (D \DP \wh D) \subset (D \DP S)$, Lemma \ref{3.2}.2 implies that $S$ is seminormal, and hence $\mathcal A (S) = \mathcal A (D)$ by 1.
\end{proof}

\medskip
\begin{lemma} \label{3.6}
Let  $D=D_1\times D_2$ be a monoid with submonoids $D_1, D_2$, and let $a=a_1a_2 \in D$ with $a_1\in D_1$ and $a_2 \in D_2$. Then $\sup \{\mathsf c_{\eq}(a_1),\mathsf c_{\eq}(a_2)\}\leq \mathsf c_{\eq}(a)$, and equality holds if $|\mathsf L (a_2)|=1$. In particular,
 $\sup\{\mathsf c_{\eq}(D_1),\mathsf c_{\eq}(D_2)\}\leq \mathsf c_{\eq}(D)$, and equality holds if $D_2$ is half-factorial.
\end{lemma}

\begin{proof}
We may assume that $D$ is reduced, and we tacitly apply \cite[Proposition 1.2.11]{Ge-HK06a}. Since $\mathsf Z(a_i) \subset \mathsf Z(a)$, it follows that $\mathsf c_{\rm eq}(a_i) \le \mathsf c_{\rm eq}(a)$ for $i \in \{1,2\}$. Assume now that $|\mathsf L(a_2)| =1$, and let $z,\,z' \in \mathsf Z(a)$ be such that $|z| = |z'| \ge 2$. Then $z = z_1z_2$ and $z' = z_1'z_2'$, where $z_i' \in \mathsf Z(a_i)$ for $i \in \{1,2\}$, hence $|z_2| =|z_2'|$ (since $|L(a_2)|=1$) and thus also $|z_1| = |z_1'|$. Let $N = \sup\{ \mathsf c_{\rm eq}(a_1),  \mathsf c_{\rm eq}(a_2)\}$. For $i \in \{1,2\}$, there exists an $N$-chain of equal-length factorizations $z_i = z_{i,1}, \ldots, z_{i,n_i} = z_i'$ in $\mathsf Z(a_i)$, and therefore \,$z =z_{1,1}z_2, \ldots, z_{1,n_1}z_2 = z_1'z_{2,1} , \ldots, z_1' z_{2,n_2} = z_1'z_2'=z'$ is an $N$-chain of equal-length factorizations concatenating $z$ and $z'$.
\end{proof}

We present the main result of the present section. It will play a central role in the proof of our main arithmetical result (see Theorem \ref{5.6}).

\medskip
\begin{theorem} \label{3.7}
Let $(D_i)_{i \in I}$ be a family of atomic monoids which are either
bifurcus or have catenary degree at most two, let $D$ be their
coproduct, and set $J = \{i \in I \mid D_i \ \text{is bifurcus} \}$.
\begin{enumerate}
\item If $J = \emptyset$, then $\mathsf c (D) \le 2$.

\smallskip
\item Suppose that $J \ne \emptyset$.
      \begin{enumerate}
      \smallskip
      \item  $\mathcal U_k (D) = \N_{\ge 2}$ for all $k \ge 2$, $\Delta (D) = \{1\}$, and
             $\mathsf c (D) = \mathsf c_{\adj} (D) = 3$.

      \smallskip
      \item If $|J| = 1$, say $J = \{j\}$, and $D_j$ is seminormal finitely primary,  then $\mathsf c_{\eq} (D) = 2$
            and $\mathsf c_{\monn} (D)  = 3$.

      \smallskip
      \item If $|J| \ge 2$, then $\mathsf c_{\monn} (D) = \mathsf c_{\eq} (D) = \sup (\{ \mathsf c_{\eq} (D_j) \mid j \in J \}\cup\{5\})$.
      \end{enumerate}
\end{enumerate}
\end{theorem}

\begin{proof}
1. Since $\mathsf c (D) = \sup \{ \mathsf c (D_i) \mid i \in I \}$
(\cite[Proposition 1.6.8]{Ge-HK06a}), it  follows that $\mathsf c
(D) \le 2$.

\smallskip
2. (a) Let $S$ be a bifurcus monoid, $u \in \mathcal A (S)$, $a \in S$, and $\ell, m \in \N$ with $2 \le \ell < m$. . Then obviously
 $\Delta (S) = \{1\}$.  Since $u^{m-\ell+2}$ can be written as a product of two atoms, $u^m = u^{m - \ell+2}u^{\ell-2}$ can be written as a product of $\ell$ atoms. Thus $\ell \in \mathcal U_m (S)$, $m \in \mathcal U_{\ell} (S)$, and thus $\mathcal U_k (S) = \N_{\ge 2}$ for all $k \ge 2$. Since any product of
three atoms can be written as a product of two atoms, there is a
$3$-chain of factorizations starting from any factorization $z \in
\mathsf Z (a)$ to a factorization $z^* \in \mathsf Z (a)$ of length
$|z^*| = 2$. Thus we get $\mathsf c (a) \le 3$ and hence $\mathsf c
(S) = 3$.

This shows that $\mathsf c (D) = 3$, and if $D_i$ is bifurcus for some $i \in I$, then
$\N_{\ge 2} = \mathcal U_k (D_i) \subset \mathcal U_k (D) \subset \N_{\ge 2}$ for all $k \ge 2$.
Furthermore, since $\sup \Delta (D) = \sup \bigcup_{i\in I} \Delta (D_i) = 1$ (\cite[Proposition 1.4.5]{Ge-HK06a}), it follows that
$\Delta (D) = \{1\}$.

It remains to show that $\mathsf c_{\adj} (D) = 3$. Let $a \in D$
with $|\mathsf L (a)| > 1$. We show that $\mathsf c_{\adj} (a) = 3$.
To do so, we pick $m \in \N$ such that $m-1, m \in \mathsf L
(a)$, and have to verify that $\dd \big( \mathsf Z_{m-1} (a),
\mathsf Z_m (a) \big) \le 3$. We have  $a = a_1 \cdot \ldots \cdot a_n$ where $i_1, \ldots, i_n \in I$ are distinct and $a_{\nu} \in
D_{i_{\nu}}$. Since
\[
\mathsf L (a) = \mathsf L (a_1) + \ldots + \mathsf L (a_n) \quad \text{and} \quad m > \min \mathsf L (a) = \sum_{\nu=1}^n \min \mathsf L (a_{\nu}) \,,
\]
there is an $\nu \in [1,n]$, say $\nu=1$, such that $m= m_1 + \ldots +
m_n$, $m_1 > \min \mathsf L (a_1)$ and $m_{\nu} \in \mathsf L (a_{\nu})$
for all $\nu \in [1,n]$. This shows that $D_{i_1}$ is not
half-factorial, hence $D_{i_1}$ is bifurcus, and for simplicity of
notation we assume that $D_{i_1}$ is reduced. Suppose that $z_1 =
u_1 \cdot \ldots \cdot u_{m_1} \in \mathsf Z (a_1)$, with $u_1,
\ldots, u_{m_1} \in \mathcal A ( {D_{i_1}})$, and $z_i \in \mathsf Z
(a_i)$ with $|z_i| = m_i$ for all $i \in [2,n]$. By assumption,
there are $v_1, v_2 \in \mathcal A (D_{i_1})$ such that $u_1u_2u_3 =
v_1v_2$, and we set $z_1' = v_1v_2u_4 \cdot \ldots \cdot u_{m_1}$.
Then $z_1' \in \mathsf Z (a_1)$, $\dd (z_1, z_1') = 3$, $z' =
z_1' z_2 \cdot \ldots \cdot z_n \in \mathsf Z_{m-1} (a)$ and
\[
\dd \big( \mathsf Z_{m-1} (a), \mathsf Z_m (a) \big) = \dd (z', z) = 3 \,.
\]

\smallskip
2.(b) Suppose that $|J| = 1$. Then 1. implies that $D = S_1 \time
S_2$ where $S_1$ is seminormal finitely primary and $\mathsf c (S_2)
\le 2$. By 2.(a) it suffices to show that $\mathsf c_{\eq} (D) \le
2$, and for that it remains to show by Lemma \ref{3.6}.2 that $\mathsf c_{\eq} (S_1) \le
2$, which follows from Lemma \ref{3.5}.

\smallskip
2.(c) Suppose that $|J| \ge 2$. By 2.(a) it remains to show that
$\mathsf c_{\eq} (D) = \sup \{ \mathsf c_{\eq} (D_j), 5 \mid j \in J
\}$. We write $D$ in the form $D = S_1 \time S_2$ where $S_1$ is the coproduct of $(D_i)_{i\in J}$ and $\mathsf c (S_2) \le
2$. By Lemma \ref{3.6}.2 it remains to show that $\mathsf c_{\eq} (S_1) =  \sup \{ \mathsf
c_{\eq} (D_j), 5 \mid j \in J \}$. First we outline that the term on
the right side is a lower bound for $\mathsf c_{\eq} (S_1)$. Let
$\alpha, \beta \in J$, and without restriction we suppose that
$D_{\alpha}$ and $D_{\beta}$ are reduced. There are atoms $a_1,
\ldots, a_5 \in \mathcal A (D_{\alpha})$ and $b_1, \ldots, b_5 \in
\mathcal A (D_{\beta})$ such that $a_1a_2 = a_3a_4a_5$ and $b_1b_2 =
b_3b_4b_5$. Then $z = a_1a_2b_3b_4b_5$ and $z' = a_3a_4a_5b_1b_2$
are factorizations of $c = a_1a_2b_1b_2 \in D_{\alpha} \time
D_{\beta} \subset S_1$ with $\mathsf c_{\eq} (c) \ge \dd (z, z')
= 5$. Since $\mathsf c_{\eq} (D_j) \le \mathsf c_{\eq} (S_1)$ for all
$j \in J$, we infer that $\sup \{ \mathsf c_{\eq} (D_j), 5 \mid j
\in J \} \le \mathsf c_{\eq} (S_1)$.

To verify equality, we assume without restriction that $c_{\eq} (D_j) < \infty$ for all $j\in J$. Let $a = a_1 \cdot \ldots \cdot a_n$, where $n
\in \N$, $a_{\nu} \in D_{i_{\nu}}$ for all $\nu \in [1,n]$
and $i_1, \ldots, i_n \in J$ distinct, $x = x_1 \cdot \ldots \cdot
x_n$, $y = y_1 \cdot \ldots \cdot y_n \in \mathsf Z (a)$, where
$x_{\nu}, y_{\nu} \in \mathsf Z (a_{\nu})$ for all $\nu \in [1, n]$,
and
\[
\sum_{\nu=1}^n |x_{\nu}| = |x| = |y| = \sum_{\nu=1}^n |y_{\nu}| \,.
\]
If $n=1$, then $a = a_1 \in D_{i_1}$ and $\mathsf c_{\eq} (a) \le
\mathsf c_{\eq} (D_{i_1}) \le \sup \{ \mathsf c_{\eq} (D_j), 5 \mid
j \in J \} $. Suppose that $n \ge 2$. We proceed by induction on
\[
\psi (x,y) = \sum_{\nu =1}^n \big| |x_{\nu}| - |y_{\nu}| \big| \,.
\]
Suppose that $\psi (x,y) = 0$. Then, for all $\nu \in [1,n]$, we
have $|x_{\nu}| = |y_{\nu}|$, and hence there is a $\mathsf c_{\eq}
(D_{i_{\nu}})$-chain of equal-length factorizations between
$x_{\nu}$ and $y_{\nu}$. Therefore, there is a $\max \{ \mathsf
c_{\eq} (D_{i_1}), \ldots, \mathsf c_{\eq} (D_{i_n})\}$-chain of
equal-length factorizations between $x$ and $y$.

Suppose that $\psi (x,y) > 0$. Then there exist $\nu, \nu' \in [1,
n]$, say $\nu = 1$ and $\nu' =2$, such that $|x_1| > |y_1|$ and
$|x_2| < |y_2|$, and suppose that $D_{i_1}, D_{i_2}$ are reduced. If $x_1 = u_1 \cdot \ldots \cdot u_k$ with $k
=|x_1|$ and $u_1, \ldots, u_k \in \mathcal A (D_{i_1})$, then $k \ge
3$, there are $u_1', u_2' \in \mathcal A (D_{i_1})$ such that
$u_1u_2u_3 = u_1'u_2'$ and $x_1' = u_1'u_2'u_4 \cdot \ldots \cdot
u_k \in \mathsf Z (a_1)$. Since $\max \mathsf L (a_2) \ge |y_2| >
|x_2|$ and $\mathsf c_{\adj} (a_2) \le 3$, there exist $x_2'', x_2'
\in \mathsf Z (a_2)$ and
$|x_2'| = |x_2''|+1= |x_2|+1$ and $\dd (x_2'', x_2') \le \mathsf c_{\adj} (a_2) \le 3$. Note that $x_2'$ arises from $x_2''$
by replacing two atoms from $x_2''$ by three new atoms.  There is  a $\mathsf c_{\eq} (D_{i_2})$-chain of
equal-length factorizations between $x_2$ and $x_2''$, and hence
there is a $\mathsf c_{\eq}(D_{i_2})$-chain of equal-length
factorizations between $x=x_1 \cdot \ldots \cdot x_k$ and $x'' =
x_1x_2''x_3 \cdot \ldots \cdot x_k$, and we set $x' = x_1'x_2'x_3
\cdot \ldots \cdot x_k$. Since $|x| =|x'|=|x''|$, $\dd (x'',
x') = \dd (x_1x_2'', x_1'x_2') = 5$ and
\[
\psi (x', y) = \big|   |x_1'| - |y_1| \big| + \big| |x_2'| - |y_2|
\big| + \sum_{\nu =3}^n \big| |x_{\nu}| - |y_{\nu}| \big| < \psi
(x,y) \,,
\]
the assertion follows from the induction hypothesis.
\end{proof}

\bigskip
\section{Class groups, transfer homomorphisms, and $T$-block monoids} \label{4}
\bigskip

We recall the concepts of abstract class groups, of
transfer homomorphisms, and  of $T$-block monoids (details can be found in \cite[Sections 2.4, 3.2, and
3.4]{Ge-HK06a}). These are the main tools for
arithmetical investigations of weakly Krull monoids. Proposition \ref{L:GlobalTransfer} will be crucial in the proof of Theorem \ref{5.6}.

\smallskip
Let $D$ be a monoid and $H \subset D$ a saturated submonoid. Then the group
\[
\cls{H}{D}= \quot(D) / \left(D^\times\quot (H)\right)
\]
is called the {\it class group} of  $H$ in $D$.  As usual we write this group additively. For $a
\in \quot(D)$, we let $[a]=[a]_{D/H}\in\cls{H}{D}$ be its image in $\cls{H}{D}$. If $D$ is reduced,
then, since $H \subset D$ is saturated, we have $H = \{a \in D \mid [a] = 0 \}$. If $P \subset D$
is a set of primes and $T \subset D$ a submonoid without primes such that $D = \mathcal F (P) \time
T$, then $\{ [p] \mid p \in P\} \subset  \cls{H}{D}$ is the set of classes containing primes. We say that $H \subset D$ is {\it cofinal}
if for
every $a \in D$ there is an $u \in H$ such that $a \t u$. If  $\mathcal H = \{aH
\mid a \in H \}$ is the monoid of principal ideals of $H$, then $\mathcal H \subset \mathcal I_v^* (H)$
is saturated and cofinal.
The map $\partial_H \colon H_\red \to \mathcal I_v^*(H)$, defined by $\partial(aH^\times) = aH$ for all $a \in H$, is a monomorphism, $\partial_H (H_\red) = \mathcal H$, and $\mathcal C_v(H) = \coker\,[\mathsf q(\partial_H) \colon \mathsf q(H)/H^\times \to \mathsf q(\mathcal I_v^*(H))] = \cls{\mathcal H}{\mathcal I_v^* (H)}$ is the
$v$-class group of $H$.

\medskip
\begin{lemma} \label{4.0}
Let $H \subset D$ be a saturated submonoid. Then the inclusion $H \hookrightarrow D$ induces an isomorphism \,$\varphi   \colon H_{\rm red} \to H^* = HD^\times/D^\times$, \ $H^* \subset D_{\rm red}$ is a saturated submonoid, and $\varphi$ induces an isomorphism $\overline \varphi \colon \mathcal C(H,D) \to \mathcal C(H^*, D_{\rm red})$ mapping the set of all classes containing primes onto the set of all classes containing primes. We will identify $\mathcal C(H,D)$ and $\mathcal C(H^*, D_{\rm red})$.
\end{lemma}

\begin{proof}
Consider the  map $H_{\red} \to D_{\red}$, defined by $aH^{\times} \mapsto a D^{\times}$ for all $a \in H$. Its restriction $\varphi \colon H_{\red} \to \varphi (H_{\red}) = \{ a D^{\time} \mid a \in H \}= HD^{\times}/D^{\times}=H^*$ is an isomorphism, $H^* \subset D_{\red}$ is saturated, $\mathsf q (H^*) = \mathsf q (H) D^{\times}/D^{\times}$, and hence $\overline \varphi$ is an isomorphism having the asserted property.
\end{proof}

\smallskip
Let $H$ and $B$ be atomic monoids and $\theta \colon H \to B$ a
monoid homomorphism. Then $\theta$ is called a \ {\it transfer
homomorphism} \ if it has the following properties:
\begin{enumerate}
\item[{\bf (T\,1)\,}] $B = \theta(H) B^\times$ \ and \ $\theta
^{-1} (B^\times) = H^\times$.

\smallskip

\item[{\bf (T\,2)\,}] If $u \in H$, \ $b,\,c \in B$ \ and \ $\theta
(u) = bc$, then there exist \ $v,\,w \in H$ \ such that \ $u = vw$,
\ $\theta (v) \simeq b$ \ and \ $\theta (w) \simeq c$.
\end{enumerate}
Clearly, $\theta$ is a transfer homomorphism if and only if $\theta_{\red}$ is a transfer homomorphism.
Suppose that $\theta$ is a transfer homomorphism. Then it gives rise
to a unique extension $\overline \theta \colon \mathsf Z(H) \to
\mathsf Z(B)$ satisfying
\[
\overline \theta (uH^\times) = \theta (u)B^\times \quad \text{for
all} \quad u \in \mathcal A(H)\,.
\]
For $a \in H$, we denote by \ $\mathsf c (a, \theta)$ \ the smallest
$N \in \N_0 \cup \{\infty\}$ with the following property:
\begin{enumerate}
\item[]
If $z,\, z' \in \mathsf Z_H (a)$ and $\overline \theta (z) =
\overline \theta (z')$, then there exist some $k \in \N_0$ and
factorizations $z=z_0, \ldots, z_k=z' \in \mathsf Z_H (a)$ such that
\ $\overline \theta (z_i) = \overline \theta (z)$ and \ $\dd
(z_{i-1}, z_i) \le N$ for all $i \in [1,k]$ \ (that is, $z$ and $z'$
can be concatenated by an $N$-chain in the fiber \ $\mathsf Z_H (a)
\cap \overline \theta ^{-1} (\overline \theta (z)$)\,).
\end{enumerate}
Then
\[
\mathsf c (H, \theta) = \sup \{\mathsf c (a, \theta) \mid a \in H \}
\in \N_0 \cup \{\infty\}\,
\]
denotes the {\it catenary degree in the fibres}.

\medskip
The following lemma gathers the properties of transfer homomorphisms
which are needed in the sequel.

\medskip
\begin{lemma} \label{4.1}
Let $\theta \colon H \to B$ be a transfer homomorphism of
atomic monoids.
\begin{enumerate}
\item For every $a \in H$, we have $\overline \theta(\mathsf Z_H(a)) = \mathsf Z_B(\theta(a))$
and \ $\mathsf L_H(a) = \mathsf L_B(\theta(a))$. In particular,
$\Delta (H) = \Delta (B)$ and $\rho_k (H) = \rho_k (B)$ for all $k
\ge 2$.

\smallskip
\item $\mathsf c (B) \le \mathsf c (H) \le \max \{ \mathsf c (B),
      \mathsf c (H, \theta) \}$, and $\mathsf c_{\monn} (B) \le \mathsf c_{\monn} (H) \le \max \{ \mathsf c_{\monn} (B),
      \mathsf c (H, \theta) \}$.

\smallskip
\item For every \ $a \in H$ \ and all \ $k, l \in \mathsf L (a)$, \ we have \
      $\dd \bigl( \mathsf Z_k (a), \mathsf Z_l (a) \bigr) =
      \dd \bigl( \mathsf Z_k \bigl( \theta(a)\bigr), \mathsf Z_l \bigl( \theta (a) \bigr)
      \bigr)$. In particular, $\mathsf c_{\adj} (a) = \mathsf
      c_{\adj} ( \theta (a) )$ and $\mathsf c_{\adj} (H) = \mathsf c_{\adj}
      (B)$.

\smallskip
\item $\mathsf c_{\eq} (B) \le \mathsf c_{\eq} (H) \le \max \{ \mathsf c_{\eq} (B), \mathsf c (H, \theta)
\}$.
\item If $\theta'\colon B\rightarrow B'$ is a second transfer homomorphism, then $\theta'\circ
\theta$ is a transfer homomorphism, such that
$\cat(H,\theta'\circ\theta)\leq\sup\{\cat(H,\theta),\cat(B,\theta')\}$.
\end{enumerate}
\end{lemma}

\begin{proof}
1. This follows from \cite[Proposition 3.2.3]{Ge-HK06a}.

\smallskip
2. and 3. This follows from \cite[Lemma 3.2]{Ge-Gr-Sc-Sc10}.

\smallskip
4. Since $|z| = |\overline \theta (z)|$ for all $z \in \mathsf Z (H)$, this is a simple consequence of \cite[Lemma 3.2.6]{Ge-HK06a}.

\smallskip
5. See \cite[Lemma 3.2]{Ge-Gr-Sc-Sc10}.
\end{proof}

\medskip
Let $G$ be an additive abelian group, $G_0 \subset G$ a
subset, $T$ a monoid and $\iota \colon T \to G$ a homomorphism. Let
$\sigma \colon \mathcal F(G_0) \to G$ be the unique homomorphism
satisfying $\sigma (g) = g$ for all $g \in G_0$. Then
\[
B = \mathcal B (G_0,T,\iota) = \{S\,t \in \mathcal F(G_0) \time T\,\mid\, \sigma (S) + \iota(t) =
0\,\} \subset \mathcal F (G_0) \time T = F
\]
the \ \textit{$T$-block monoid over \ $G_0$ \ defined by \
$\iota$}\,, and
\[
\mathcal B (G_0) = \{ S \in \mathcal F (G_0) \mid \sigma (S) = 0 \} \subset \mathcal F (G_0)
\]
is the {\it monoid of zero-sum sequences} over $G_0$. Then $B \subset F$ is saturated, and if
$G_0 = G$ or if $G_0 \cup \iota (T)$ contains only torsion elements,
then $B \subset F$ is cofinal.
  If $T$ is reduced, then $\mathcal B (G_0) \subset B$ is a divisor-closed submonoid whence $\mathsf Z_B (A) = \mathsf Z_{\mathcal B (G_0)} (A)$ for all $A \in \mathcal B (G_0)$, and if $T = \{1\}$, then
$\mathcal B (G_0) = B$.

\medskip
\begin{lemma} \label{4.2}
Let $D = \mathcal F (P) \time T$ be a reduced atomic monoid, where $P \subset D$ a set of primes
and $T \subset D$ is a  submonoid, and let $H \subset D$ be an atomic saturated submonoid with
class group $G = \mathcal C (H, D)$, and  $ G_P = \{[p] \mid p \in P\} \subset G$ the set of classes containing primes.
Let $\iota \colon T \to G$ be defined by $\iota (t) = [t]$, \ $F= \mathcal F (G_P) \time T$, \ $B =
\mathcal B(G_P,T, \iota) \subset F$, and let $\widetilde{\boldsymbol \beta} \colon D \to F$ be the
unique homomorphism satisfying $\widetilde{\boldsymbol \beta} (p) = [p]$ for all $p \in P$ and
$\widetilde{\boldsymbol \beta} \t T = \text{\rm id}_T$.

\begin{enumerate}
\item The restriction $\boldsymbol \beta = \widetilde{\boldsymbol
\beta} \t H \colon H \to B$ \ is a transfer homomorphism satisfying
$\mathsf c (H, \boldsymbol \beta) \le 2$.

\item If $H \subset D$ is cofinal, then $B \subset F$ is cofinal, and there is an isomorphism $\overline \psi \colon \cls{B}{F} \to G$,
such that $\overline{\psi}(S\, t)=\sigma(S)+\iota(t)$ for all $S\, t\in \mathcal F(G_P) \time T$ by which we will identify these groups.
\end{enumerate}
\end{lemma}

\begin{proof}
See \cite[Proposition 3.4.8]{Ge-HK06a}.
\end{proof}

\medskip
\begin{lemma}\label{L:IndTransfer}
Let $\theta\colon D\rightarrow \overline{D}$ be a transfer homomorphism and $\overline{H}\subset \overline{D}$ a
saturated submonoid such that $\overline{D}^\times=\theta(D^\times)\overline{H}^\times$. If
$H=\theta^{-1}(\overline{H})$ and $\theta'\colon H\rightarrow \overline{H}$ is induced by $\theta$, then
$\theta'$ is a transfer homomorphism.
\end{lemma}

\begin{proof}
By assumption, $\overline D = \theta(D) \overline D^\times$, \ $\theta^{-1}(\overline D^\times ) = D^\times$, and $\overline D^\times = \theta(D^\times) \overline H^\times$, whence also $\overline D = \theta(D) \overline H^\times$. Since $\overline H \subset \overline D$ is saturated, it follows that $H \subset D$ is saturated, and $\mathsf q(\overline H) \cap \overline D^\times = \overline H^\times$. We have to verify the properties \textbf{(T\,1)} and \textbf{(T\,2)}.

If $x \in \overline H \subset \overline D = \theta(D) \overline H^\times$, then $x = \theta(a) \varepsilon$, where $a \in D$ and $\varepsilon \in \overline H^\times$. Since $\theta(a) = x \varepsilon^{-1} \in \overline H$, it follows that $a \in \theta^{-1}(\overline H) = H$ and $x \in \theta(H) \overline H^\times$. Hence $\overline H = \theta(H) \overline H^\times$.

If $a \in H$ and $\theta(a) \in \overline H^\times \subset \overline D^\times$, then $a \in H \cap \theta^{-1}(\overline D^\times) = H \cap D^\times = H^\times$. Hence it follows that $\theta^{-1}(\overline H^\times) = H^\times$, and thus \textbf{(T\,1)} holds.

In order to verify \textbf{(T\,2)}, let $u \in H$ and $b,\,c \in \overline H$ are such that $\theta(u) = bc$. Then there exist $v_1,\,w_1 \in \overline D$ and $\varepsilon_1,\, \eta_1 \in \overline D^\times$ such that $u = v_1w_1$, \ $\theta(v_1) = b \varepsilon_1$ and $\theta(w_1) = c \eta_1$. We set $\varepsilon_1 = \theta(e)\varepsilon$, where $e \in D^\times$ and $\varepsilon \in \overline H^\times$. Then $\theta(e^{-1}v_1) = b \varepsilon \in \overline H$, hence $v = e^{-1} v_1 \in H$, \ $\theta(v) = b\varepsilon$ and $u = v(ew_1)$. Since $H \subset D$ is saturated, it follows that $w=ew_1 \in H$, and $\theta(w) = \varepsilon_1 \varepsilon^{-1} \eta_1c$. Since $\eta = \varepsilon_1 \varepsilon^{-1} \eta_1 = c^{-1} \theta (w) \in \mathsf q(\overline H) \cap \overline D^\times = \overline H^\times$,  property \textbf{(T\,2)} holds.
\end{proof}

\medskip

\begin{lemma}\label{L:LocalTransfer}
Let $D$ be a seminormal finitely primary monoid of rank one. Then the inclusion $i\colon
D\hookrightarrow \wh{D}$ is a transfer homomorphism.
\end{lemma}

\begin{proof}
Since $D$ is seminormal, by Lemma \ref{3.3}.1, we have $\widehat D = F = F^{\times} \time [q]$ and $D=D'=qF \cup D^{\times}$. So  $\wh{D}^\times D = F^{\times}D=\wh D$. From $D^\times
=D\cap \wh{D}^\times$ we obtain \textbf{(T\,1)}.

In order to show \textbf{(T\,2)}, let $u\in D$ and $b$, $c \in \wh{D}$ be such
that $u=bc$. Set $b=\ve q^n$, $c=\delta q^m$, where $\ve$,
$\delta\in\wh{D}^\times$ and $n$, $m \in\N_0$. If $n=0$ or $m=0$, say $n=0$, then $u=1u$ and
so $1$ and $b$ are associated in $\wh{D}$, and so are $u$ and $c$ . Now suppose $n\geq 1$ and $m\geq 1$.
Then $b$, $c\in D$ and the assertion is clear.
\end{proof}

\medskip
\begin{proposition}\label{L:GlobalTransfer}
Let $n\in\N$, $D_0$ be an atomic monoid and $D_1$,\ldots , $D_n$ be seminormal finitely primary
monoids of rank one. Set $D=D_0\time D_1\time\ldots \times D_n$ and
$\overline{D}=D_0\times\wh{D}_1 \times\ldots\time\wh{D}_n$. Suppose that
$\overline{H}\subset\overline{D}$ is a saturated submonoid such that $\overline{D}^\times=\overline{H}^\times D^\times$
and set $H=D\cap\overline{H}$. Then the inclusion $i\colon H\hookrightarrow \overline{H}$ is a transfer
homomorphism with $\cat(H,i)\leq 2$.
\end{proposition}

\begin{proof}
We reduce  to the case $n=1$ as follows. For $k \in [0, n]$,  set $\overline{D}_k=D_0\times
\wh{D}_1\times\ldots \wh{D}_k\times D_{k+1}\time\ldots\times D_n$ and
$\overline{H}_k=\overline{H}\cap \overline{D}_k$, and note that $\overline{H}_k \subset  \overline{D}_k$ is saturated. For $k \in [0, n-1]$, let $i_k\colon
\overline{H}_{k}\hookrightarrow\overline{H}_{k+1}$ be the inclusion. Then $H=\overline{H}_0$, $\overline{H}=\overline{H}_n$, and
$i=i_{n-1}\circ\ldots \circ i_0$. By Lemma \ref{4.1}.5, it suffices to show that
 $i_k$ is a transfer homomorphism with $\cat(\overline{H}_k,i_k)\leq 2$ for each $k \in [0, n-1]$.

We fix $k \in [0, n-1]$. First we show
$\overline{D}_{k+1}^\times=\overline{D}_k^\times\overline{H}_{k+1}^\times$.  Since $D^\times\subset
\overline{D}_{k+1}^\times$ we have by assumption
\begin{gather*}
\overline{D}_{k+1}^\times=\overline{D}_{k+1}^\times\cap\overline{D}^\times=\overline{D}_{k+1}^\times\cap\left(\overline{H}^\times
D^\times\right)=\left(\overline{D}_{k+1}^\times\cap
\overline{H}^\times\right)D^\times=\\
=\left(\overline{D}_{k+1}\cap\overline{H}\right)^\times D^\times=\overline{H}_{k+1}^\times D^\times\subset
\overline{H}_{k+1}^\times \overline{D}_k^\times\quad ,
\end{gather*}
and the converse inclusion is trivial.

If $D_k' = D_0 \time \wh D_1 \time \ldots \time \wh D_k \time D_{k+2} \time \ldots \time D_n$, then $\overline D_k = D_k' \time D_{k+1}$, \ $\overline D_{k+1} = D_k' \time \wh D_{k+1}$, and if we apply the assertion for $n = 1$ to $(D_k', D_{k+1})$ instead of $(D_0,D_1)$, it follows that $i_k \colon \overline H_k \hookrightarrow \overline H_{k+1}$ is a transfer homomorphism satisfying $\mathsf c(\overline H_k, i_k) \le 2$.

We proceed with the case $n=1$, and we may assume that $D = D_0 \time D_1$ is reduced. Then $D_0, \,D_1$ and $H$ are also reduced, and we set $\wh D_1 = \wh D_1^\times \time [q]$. By Lemma \ref{L:LocalTransfer}, the inclusion   $D_1 \hookrightarrow \wh D_1$ and thus also the inclusion $D_0 \time D_1 \hookrightarrow D_0 \time \wh D_1$ is a transfer homomorphism. By Lemma \ref{L:IndTransfer} also $i \colon H \hookrightarrow \overline H$ is a transfer homomorphism, and it remains to prove that $\mathsf c(H,i) \le 2$. For this, we fix some $a \in H$, \ $w \in \mathsf Z_{\overline H}(a)$ and $z,\,z' \in \mathsf Z_H(a)$ such that $\overline i(z) = \overline i(z') = w$. Then $|z| = |z'| = |w| = m \in \N_0$, and we must show that there exist factorizations $z=z_0, z_1, \ldots, z_r = z'$ for some $r \in \N_0$ such that $z_i \in \mathsf Z_H(a)$, \ $\overline i(z_i) = w$ and $\mathsf d(z_{i-1}, z_i) \le 2$ for all $ i \in [1,r]$. We use induction on $m$. If $m=0$, then $z=z'=1$, and there is nothing to do. Thus suppose that $m>0$, \ $z = u_1 \cdot \ldots \cdot u_m$ and $z' = u_1' \cdot \ldots \cdot u_m'$, where $u_j,\,u_j' \in \mathcal A(H) \subset D_0 \time D_1 \subset D_0 \time \wh D_1$, say $u_j = d_j\delta_j q^{e_j}$ and $u_j' = d_j'\delta_j' q^{e_j'}$ with $d_j,\,d_j' \in D_0$, \ $\delta_j,\,\delta_j' \in \wh D_1^\times$ and $\delta_j q^{e_j},\, \delta_j' q^{e_j'} \in D_1$ for all $j \in [1,m]$. After renumbering if necessary, we may assume that $u_j = \varepsilon_ju_j'$, where $\varepsilon_j \in \overline H^\times \subset (D_0 \time \wh D_1)^\times = \wh D_1^\times$, and therefore $d_j\delta_j q^{e_j} = d_j'(\varepsilon_j\delta_j')q^{e_j'}$, hence $d_j = d_j'$, \ $\delta_j = \varepsilon_j\delta_j'$ and $e_j= e_j'$ for all $j \in [1,m]$.

Suppose first that $e_j=0$ for some $j \in [1,m]$, say for $j=m$. Then $\delta_m,\,\delta_m' \in D_1^\times = \{1\}$, hence $\varepsilon_m=1$ and $u_m=u_m'$.
Now set $\overline{z}=u_1 \cdot \ldots \cdot u_{m-1}$, $\overline{z}'=u'_1 \cdot \ldots \cdot u'_{m-1}$ and $\overline{a}=au_m^{-1}$. Then
$\overline{z}$, $\overline{z}'\in \ZZ_H(\overline{a})$, $\overline{i}(\overline{z})=\overline{i}(\overline{z}')=:\overline{w}$ and
$|\overline{w}|<|w|$. By the induction hypothesis $\overline{z}$ and $\overline{z}'$ can be concatenated by a
$2$-chain $\overline{z}_0$,\ldots ,$\overline{z}_r$ such that $\overline{z}_l\in\ZZ_H(\overline{a})$ and
$\overline{i}(\overline{z}_l)=\overline{w}$ for $l \in [0, r]$. Then
\[
\overline{z}_0u_m,\ldots ,\overline{z}_ru_m
\]
is a $2$-chain concatenating $z$ and $z'$, such that $\overline{z}_lu_m\in \ZZ_H(a)$ and
$\overline{i}(\overline{z}_lu_m)=w$ for $l \in [0, m]$.

We may therefore assume $e_k>0$ for all $k \in [1, m]$. We set
$\overline{u}'_{m-1}=\ve_m^{-1}u'_{m-1}$ and we claim $\overline{u}'_{m-1}\in\mathcal{A}(H)$. Since
$\ve_m\in\overline{H}^\times$ we have $\overline{u}'_{m-1}\in\overline{H}$. Since $D_0$ is reduced we have
$\ve_m\in \wh{D}_1^\times$ . Then, since $D_1$ is seminormal and $e_{m-1}>0$, we obtain
\[
\overline{u}'_{m-1}=d_{m-1}\ve_m^{-1}\delta'_{m-1}q^{e_{m-1}}\in D\quad .
\]
Hence $\overline{u}'_{m-1}\in\overline{H}\cap D=H$. By construction $\overline{u}'_{m-1}$ and $u'_{m-1}$ are
associated in $\overline{H}$. Since $u'_{m-1}\in\mathcal{A}(H)$ and the inclusion
$H\hookrightarrow\overline{H}$ is a transfer homomorphism, we obtain $\overline{u}'_{m-1}\in\mathcal{A}(H)$.

By construction we have $\overline{u}'_{m-1}u_m=u'_{m-1}u'_m$. Hence $z''=u'_1 \cdot \ldots \cdot
u'_{m-2}\overline{u}'_{m-1}u_m\in \ZZ_H(a)$,  $\dd(z'',z')\leq 2$ and $\overline{i}(z'')=w$. Hence it is
sufficient to concatenate $z$ and $z''$ by a $2$-chain $z_0,\ldots , z_r$ such that $z_l\in\ZZ_H(a)$
and $\overline{i}(z_l)=w$, for $l \in [0, r]$. Since $u_m\mid z$ and $u_m\mid z''$, this follows from
the induction hypothesis.
\end{proof}

\bigskip
\section{Seminormal weakly Krull monoids and domains} \label{5}
\bigskip

Weakly Krull domains were introduced by Anderson, Anderson, Mott, and Zafrullah (\cite{An-An-Za92b, An-Mo-Za92}), and a divisor theoretic characterization was first given by Halter-Koch \cite{HK95a}. Examples will be discussed in \ref{5.7}. The main results of this section (and indeed of the present paper) are Theorem \ref{5.5} and Theorem \ref{5.6}, and they deal with seminormal $v$-noetherian weakly Krull monoids having a nontrivial conductor. Clearly, every Krull monoid is  seminormal $v$-noetherian weakly Krull, and for Krull monoids the results are well-known.

We gather the algebraic properties of weakly Krull monoids which are needed for our arithmetical investigations. Our main reference is \cite[Chapters 21 -- 24]{HK98}.
Let $\varphi \colon H \to D$ be a monoid homomorphism. We set $H_{\varphi} = \{ a^{-1}b \in \mathsf q (H) \mid a, b \in H, \varphi (a) \mid_D \varphi (b)\}$, and we say that $\varphi$ is a {\it divisor homomorphism} if, for all $a, b \in H$, $\varphi (a) \mid_D \varphi (b)$ implies that $a \mid_H b$ (equivalently, $H_{\varphi} = H$).
A family of monoid homomorphisms
$\boldsymbol \varphi = (\varphi_p \colon H \to D_p)_{p \in P}$ is said to be
\begin{itemize}
\item {\it of finite character} if the set  $\{p \in P \mid \varphi_p (a) \notin D_p^{\times}\}$ is finite for all $a \in H$.

\item a {\it  defining family} (for $H$) if it is of finite character and
      \[
      H = \bigcap_{p \in P} H_{\varphi_p} \,.
      \]
\end{itemize}
If $\boldsymbol \varphi$ is of finite character, then it induces a monoid homomorphism
\[
 \varphi \colon H \ \to \ D = \coprod_{p \in P}(D_p)_{\red} \qquad \text{defined by} \qquad \varphi(a) =
  (\varphi_p(a)D_p^\times)_{p \in P}\,,
\]
and $\boldsymbol \varphi$ is a defining family if and only if $\varphi$ is a divisor homomorphism.
A monoid $H$ is called a {\it weakly Krull monoid} (\cite[Corollary 22.5]{HK98}) if the family of embeddings $(\varphi_{\mathfrak p} \colon H \hookrightarrow H_{\mathfrak p})_{\mathfrak p \in \mathfrak X (H)}$ is a defining family of $H$, in other words
\[
H = \bigcap_{{\mathfrak p} \in \mathfrak X (H)} H_{\mathfrak p}  \quad \text{and} \quad \{{\mathfrak p} \in \mathfrak X (H) \mid a \in {\mathfrak p}\} \quad \text{is finite for all} \ a \in H \,.
\]
The class group
\[
\mathcal C \Big( \varphi(H),  \coprod_{{\mathfrak p} \in \mathfrak X (H)} (H_{\mathfrak p})_{\red} \Big) = \coker \Bigl( \mathsf q(\varphi) \colon \mathsf q(H) \to \coprod_{\mathfrak p \in \mathfrak X(H)} \mathsf q(H)/H_\mathfrak p^\times \Bigr)
\]
of $\varphi(H)$ in  $\coprod_{{\mathfrak p} \in \mathfrak X (H)} (H_{\mathfrak p})_{\red}$ depends only on $H$. It is called the {\it weak divisor class group} of $H$ and will be denoted by $\mathcal C (H)$.
A monoid $H$ is called
\begin{itemize}
\item {\it weakly factorial} if  every nonunit is a finite product of primary elements. Obviously, every primary monoid is weakly factorial, and every coproduct of a family of weakly factorial monoids is weakly factorial (see also Examples \ref{5.7}.3.)

\item {\it Krull} if it is weakly Krull and  $H_{\mathfrak p}$ is a discrete valuation monoid for each $\mathfrak p \in \mathfrak X (H)$. Since a monoid is Krull if and only if if it is $v$-noetherian and     completely integrally closed, we obtain that  every Krull monoid is a seminormal $v$-noetherian weakly Krull monoid.
\end{itemize}
A domain $R$ is called a {\it $($weakly$)$ Krull domain} ({\it weakly factorial}, resp.) if $R^{\bullet}$ is  (weakly) Krull (weakly factorial, resp.). Note that $H$ is (weakly) Krull if and only if $H_{\red}$ is (weakly) Krull. Suppose that $H$ is $v$-noetherian. Then $H$
is weakly Krull if and only if $v$-$\max (H)
= \mathfrak X (H)$ (\cite[Theorem 24.5]{HK98}), and $H$ is weakly factorial if and only if it is weakly Krull and $\mathcal C_v (H) = 0$.
If   $\emptyset\not=X\subset \mathsf q (H)$ is a fractional subset  and $S \subset H$ is a submonoid, then, by \cite[Theorem 2.3.5 and Proposition 2.2.8]{Ge-HK06a},
\[
\wh{S^{-1} H} = S^{-1} \wh H, \quad S^{-1}(H \DP X)=(S^{-1}H \DP S^{-1}X) \quad \text{and} \quad  S^{-1}(X_v)=(S^{-1}H \DP (S^{-1}H \DP S^{-1}X)) \,.
\]

\medskip
\begin{lemma} \label{5.1}
Let $H$ be a  weakly Krull monoid.
\begin{enumerate}
\item $H$ is $v$-noetherian if and only if $H_{\mathfrak p}$ is $v$-noetherian for each ${\mathfrak p} \in \mathfrak X (H)$.

\smallskip
\item $H$ is seminormal if and only if $H_{\mathfrak p}$ is seminormal for each ${\mathfrak p} \in \mathfrak X (H)$.

\smallskip
\item Suppose that  $H$ is $v$-noetherian. For $\mathfrak p \in \mathfrak X(H)$, let $S(\mathfrak p)$ be the set of all $\mathfrak P \in \mathfrak X(\wh H)$ lying above $\mathfrak p$, and set $|S(\mathfrak p)| = s_\mathfrak p$. Then
      \[\qquad
      S(\mathfrak p) = \{ \mathfrak P \in \mathfrak X(\wh H) \mid \mathfrak P \cap H = \mathfrak p\} = \{ \mathfrak q \cap \wh H \mid \mathfrak q \in \mathfrak X(\wh{H_\mathfrak p})\} \,.
      \]
      In particular, the map \,$\mathfrak X(\wh H) \to \mathfrak X(H)$, \ $\mathfrak P \mapsto \mathfrak P \cap H$, is surjective, and it is bijective if and only if $s_\mathfrak p = 1$ for all $\mathfrak p \in \mathfrak X(H)$. If $\mathfrak p \in \mathfrak X(H)$ and $H_{\mathfrak p}$ is finitely primary, then it has rank   $s_{\mathfrak p}$.

\smallskip
\item If $S\subset H$ is a saturated submonoid such that $\mathcal C(S, H)$ is a torsion group, then $S$ is  weakly Krull.
\end{enumerate}
\end{lemma}

\begin{proof}
1. and 2. If $H$ is seminormal \,[$v$-noetherian], then every localization has the same property (see Lemma \ref{3.1} and \cite[Proposition 2.2.8]{Ge-HK06a}). Conversely, if $H_\mathfrak p$ is seminormal \,[$v$-noetherian] for all $\mathfrak p \in \mathfrak X(H)$, the same is true their coproduct
\[
P = \coprod_{\mathfrak p \in \mathfrak X(H)} H_{\mathfrak p}
\]
and for every saturated submonoid of $P$ (see Lemma \ref{3.1} and \cite[Propositions 2.1.11 and 2.4.4]{Ge-HK06a}).
Since there is a divisor homomorphism \,$\varphi \colon H \to P$, it follows that $H_\red$ is isomorphic to a saturated submonoid of $P$. Hence $H_\red$ and thus also $H$ is seminormal \,[$v$-noetherian].

\smallskip

3. Let $\mathfrak p \in \mathfrak X(H) = v\text{-}\max(H)$ and $S = H \setminus \mathfrak p$. Since $S^{-1} \wh H = \wh{S^{-1}H} = \wh{H_\mathfrak p}$ and $\mathfrak P \cap H \in \mathfrak X(H)$ for all $\mathfrak P \in \mathfrak X(\wh H)$, it follows that
\[
S(\mathfrak p) = \{ \mathfrak P \in \mathfrak X(\wh H) \mid \mathfrak P \cap H = \mathfrak p\} = \{ \mathfrak P \in \mathfrak X(\wh H) \mid \mathfrak P \cap H \supset \mathfrak p\} = \{ \mathfrak q \cap \wh H \mid \mathfrak q \in \mathfrak X(\wh{H_\mathfrak p}) \}
\]
is the set of all minimal non-empty prime $s$-ideals of $\wh H$ lying above $\mathfrak p$, and since the map
\[
\mathfrak X(\wh{H_\mathfrak p}) \to S(\mathfrak p) , \quad \text{defined by} \quad \mathfrak q \to \mathfrak q \cap \wh H
\]
is bijective, it follows that $|S(\mathfrak p)| = |\mathfrak X(\wh{H_\mathfrak p})|$.

In particular, the map $\pi \colon \mathfrak X(\wh H) \to \mathfrak X(H)$, defined by $\pi(\mathfrak P) = \mathfrak P \cap H$, is surjective, and $|\pi^{-1}(\mathfrak p)| = s_\mathfrak p$ for all $\mathfrak p \in \mathfrak X(H)$. If $H_{\mathfrak p}$ is finitely primary, then it has rank $|\mathfrak X(\wh{H_\mathfrak p})| = s_\mathfrak p$.

\smallskip
4. First suppose that $H$ is reduced. Let $y\in H$. Some $k\in\N$ exists such that $(y \mathsf q (S))^k= \mathsf q (S)$, hence $y^k\in  \mathsf q (S)\cap H=S$. Therefore, $S\subset H$ is a root extension, and thus \cite[Proposition 2.7.2]{Re13a} implies that $\{\mathfrak p\cap S\mid \mathfrak p\in\mathfrak{X}(H)\}=\mathfrak{X}(S)$. Let $x\in S\setminus S^{\times}$. Then $x\in H\setminus H^{\times}$ and by \cite[Theorem 22.7]{HK98} there are some $n\in\N$ and some sequence $(\mathfrak q_i)_{i=1}^n$ of primary ideals of $H$ such that $xH=\bigcap_{i=1}^n \mathfrak q_i$ and $\sqrt[H]{\mathfrak q_j}\in\mathfrak{X}(H)$ for all $j\in [1,n]$. Since $S\subset H$ is saturated, this implies that $xS=xH\cap S=\bigcap_{i=1}^n (\mathfrak q_i\cap S)$. Observe that $\mathfrak q_i\cap S$ is a primary ideal of $S$ and $\sqrt[S]{\mathfrak q_i\cap S}=\sqrt[H]{\mathfrak q_i}\cap S\in\mathfrak{X}(S)$ for all $i\in [1,n]$. By \cite[Theorem 22.7]{HK98} we obtain that $S$ is a weakly Krull monoid.

Now  we consider the general case. Clearly, $H_{\red}$ is weakly Krull and hence $S_{\red}$ is weakly Krull by Lemma \ref{4.0} and by the above arguments. Therefore $S$ is weakly Krull.
\end{proof}

\medskip
Saturated submonoids of  monoids $D = F \times (D_1 \time \ldots \time D_n)$, where $F$ is factorial and $D_1, \ldots, D_n$ are primary,  play an important role as  auxiliary monoids in factorization theory (see, for example, \cite[Section 4.5]{Ge-HK06a}. Note that  $T$-block monoids, as  introduced in Section \ref{4}, have such a form in case where $T=D_1 \time \ldots \time D_n$). We will need such monoids in the proofs of Theorems \ref{5.6} and \ref{6.2}. It has not been observed so far that they are weakly Krull monoids. The following lemma provides the simple argument.

\medskip
\begin{lemma} \label{5.2}
Let $D = \mathcal F (P) \time \prod_{i=1}^n D_i$ be a monoid, where $P \subset D$ is a set of primes, $n \in \N_0$, and $D_1, \ldots, D_n$ are primary monoids, and let $H \subset D$ be a saturated submonoid.
\begin{enumerate}
\item If $\mathcal C(H, D)$ is a torsion group, then $H$ is a weakly Krull monoid.

\smallskip
\item If $D_1,...,D_n$ are seminormal finitely primary, then $H$ is  seminormal $v$-noetherian with $(H \DP \wh H) \ne \emptyset$.
\end{enumerate}
\end{lemma}

\begin{proof}
1. Since $D$ is  weakly factorial, it is weakly Krull, and hence  the assertion follows from Lemma \ref{5.1}.4.

\smallskip
2. Since $\mathcal F (P), D_1, \ldots, D_n$ are all seminormal, $H$ is seminormal by Lemma \ref{3.1}.
Since $D_1, \ldots, D_n$ are $v$-noetherian by Lemma \ref{3.3}, $D$ is $v$-noetherian, and hence $H$ is $v$-noetherian as a saturated submonoid. Furthermore, we have
\[
\wh D = \mathcal F (P) \time \wh{D_1}\time \ldots \time \wh{D_n}
\]
and since $(D_i \DP \wh{D_i}) \ne \emptyset$, it follows that $(D \DP \wh D) \ne \emptyset$. Again, since $H \subset D$ is saturated, $(D \DP \wh D) \ne \emptyset$ implies that $(H \DP \wh H) \ne \emptyset$ by \cite[Lemma 3.5]{Ge-Ha08b}.
\end{proof}

\medskip
\begin{proposition} \label{5.3}
Let $H$ be a $v$-noetherian weakly Krull monoid.

\begin{enumerate}
\item
If \,$\mathfrak a \in \mathcal I_v(H)$, then $\mathfrak a$ is $v$-invertible if and only if $\mathfrak a_\mathfrak p$ is principal for all $\mathfrak p \in \mathfrak X(H)$.

\smallskip

\item
If $\mathfrak p \in \mathfrak X(H)$, then $\mathcal I_v^*(H_\mathfrak p) = \{aH_\mathfrak p \mid a \in H\}$, and the map $(H_\mathfrak p)_\red \to \mathcal I_v^*(H_\mathfrak p)$, defined by $aH_\mathfrak p^\times \mapsto aH_\mathfrak p$ for all $a \in H$,  is a monoid isomorphism.

\smallskip

\item
If $a \in H_\mathfrak p$, then $aH_\mathfrak p \cap H \in \mathcal I_v^*(H)$, \ $(aH_\mathfrak p \cap H)_\mathfrak p = aH_\mathfrak p$, and if \,$\mathfrak q \in \mathfrak X(H) \setminus \{ \mathfrak p\}$, then $aH_\mathfrak p \cap (H \setminus \mathfrak q) \ne \emptyset$ and $(aH_\mathfrak p \cap H)_\mathfrak q = H_\mathfrak q$.

\smallskip

\item
The map
\[
\delta_H \colon \mathcal I_v^*(H) \to \coprod_{\mathfrak p \in \mathfrak X(H)} (H_\mathfrak p)_\red\,,
\]
defined by $\delta_H (\mathfrak a)(\mathfrak p) = a_{\mathfrak p}H_\mathfrak p^\times$ if $\mathfrak a_\mathfrak p = a_{\mathfrak p}H_\mathfrak p$, is a monoid isomorphism and induces an isomorphism \,$\overline \delta_H \colon \mathcal C_v(H) \to \mathcal C(H)$. Moreover, $\mathcal I_v^*(H)$ is $v$-noetherian and weakly factorial.
\end{enumerate}
\end{proposition}

\begin{proof}
1. See \cite[Theorem 2.2.10]{Ge-HK06a}.

\smallskip

2. If $\mathfrak p \in \mathfrak X(H)$, then $H_\mathfrak p$ is $v$-local by \cite[Proposition 2.2.8.6]{Ge-HK06a}, and thus the assertion follows by \cite[Theorem 2.2.5.4]{Ge-HK06a}.

\smallskip

3. The first two assertions follow by \cite[Proposition 2.2.8.3]{Ge-HK06a}. Let now $\mathfrak p,\,\mathfrak q \in \mathfrak X(H)$ and $a \in H_\mathfrak p$ be such that $aH_\mathfrak p \cap (H \setminus \mathfrak q) = \emptyset$. Then $a \notin H_\mathfrak p^\times$, hence $a \in \mathfrak pH_\mathfrak p$ and $aH_\mathfrak p \cap H \subset \mathfrak p \cap \mathfrak q$. Since $H \setminus (\mathfrak p \cap \mathfrak q)$ is multiplicatively closed, \cite[Theorem 2.2.4]{Ge-HK06a} implies that there exists some $\overline{\mathfrak p} \in \mathfrak X(H)$ such that $aH_\mathfrak p \cap H \subset \overline{\mathfrak p} \subset  \mathfrak p \cap \mathfrak q$. Since $\mathfrak X(H) = v\text{-}\max(H)$, it follows that $\mathfrak p = \mathfrak q = \overline{\mathfrak p}$. In particular, if $\mathfrak p \ne \mathfrak q$, then $aH_\mathfrak p \cap (H \setminus \mathfrak q) \ne \emptyset$, and therefore $(aH_\mathfrak p \cap H)_\mathfrak q = H_\mathfrak q$.

\smallskip

4. If $\mathfrak p,\, \mathfrak q \in \mathfrak X(H)$ and $\mathfrak p \ne \mathfrak q$, then $aH_\mathfrak p \cap H_\mathfrak q^\times \cap H \ne \emptyset$ for all $a \in H \subset H_\mathfrak p$. Hence $(H_\mathfrak p)_{\mathfrak p \in \mathfrak X(H)}$ is a divisorial defining system for $H$, and $\delta_H$ is an isomorphism by \cite[Theorem 26.5]{HK98}. Since the coproduct of a family of $v$-noetherian weakly factorial monoids is $v$-noetherian and weakly factorial, $\mathcal I_v^*(H)$ is also $v$-noetherian and weakly factorial.

By definition,
\[\mathcal C_v(H) = \coker \bigl( \mathsf q( \partial_H)\bigr), \quad \mathcal C(H) = \coker \Bigl( \mathsf q(H) \to \coprod_{\mathfrak p \in \mathfrak X(H)} \mathsf q(H)/H_\mathfrak p^\times \Bigr),
\]
and there is an isomorphism
\[
\mathsf q(\delta_H) \colon \mathsf q(\mathcal I_v^*(H)) \to \coprod_{\mathfrak p \in \mathfrak X(H)} \mathsf q(H)/H_\mathfrak p^\times
\]
which satisfies $\mathsf q(\delta_H) (aH) = (aH_\mathfrak p^\times)_{\mathfrak p \in \mathfrak X(H)}$ for all $a \in \mathsf q(H)$ and induces an isomorphism $\overline \delta_H \colon \mathcal C_v(H) \to \mathcal C(H)$.
\end{proof}

\medskip
Based on the last proposition we  deduce an exact sequence
involving the class groups of $H$ and $\wh{H}$. The existence of such an exact sequence is a classical result for  one-dimensional noetherian
domains, and it was first generalized to arbitrary noetherian domains in \cite[Theorem 3]{Ka99b}. Further variants in more general settings can be found in \cite[Chapter 26]{HK98} and in  \cite[Theorem 2.6.7]{Ge-HK06a}.

\medskip
\begin{proposition} \label{exactsequence}
Let $H$ be a  $v$-noetherian weakly Krull monoid with $\emptyset \ne \mathfrak f = (H \DP \wh H) \subsetneq H$
such that  $H_{\mathfrak{p}}$ is finitely primary for each $\mathfrak{p}\in\mathfrak{X}(H)$. Then there is an exact sequence
\[
1 \to
\wh{H}^\times/H^\times  \to \coprod_{\mathfrak{p}\in
\mathfrak{X}(H)}\wh{H}_{\mathfrak{p}}^\times/H^\times_{\mathfrak{p}} \overset{\ve} \to
\mathcal{C}_v(H)\overset{\overline{\vartheta}} \to \mathcal{C}_v(\wh{H}) \to 0   \,,
\]
where $\varepsilon$ is given by
\[
\varepsilon = \overline \delta^{-1}_H \circ \pi \Bigm| \coprod_{\mathfrak p \in \mathfrak X(H)} \wh H_\mathfrak p^\times/H_\mathfrak p^\times, \quad  \quad \pi \colon \coprod_{\mathfrak p \in \mathfrak X(H)} \mathsf q(H_\mathfrak p)/H_\mathfrak p^\times \to \mathcal C(H) \quad \text{is canonical }
\]
and \,$\overline \delta_H \colon \mathcal C_v(H) \to \mathcal C(H)$ is the isomorphism given in Proposition \ref{5.3}.4. If \,$C =[\mathfrak a] \in \mathcal C_v(H)$, where $\mathfrak a \in \mathcal I_v^*(H)$, then $\overline \vartheta (C) = [\mathfrak a_{v(\wh H)}] \in \mathcal C_v(\wh H)$.
\end{proposition}

\begin{proof}
By \cite[Theorem 2.3.5.1]{Ge-HK06a}, $\wh H$ is a fractional $v$-ideal of $H$. Hence, if $\mathfrak a$ is a $v$-ideal of $\wh H$, then $\mathfrak a \cap H = (\wh H \DP (\wh H \DP \mathfrak a))\cap H$ is a $v$-ideal of $H$. If we temporarily denote the $v$-operations on $H$ and $\wh H$ by $v(H)$ and $v(\wh H)$, then it follows by \cite[Propositions 5.4 and 5.5]{HK98} that the inclusion $H \hookrightarrow \wh H$ is a $(v(H),v(\wh H))$-homomorphism and induces a monoid homomorphism $\vartheta' \colon \mathcal I_v(H) \to \mathcal I_v(\wh H)$, given by $\vartheta'(\mathfrak a) = \mathfrak a_{v(\wh H)}$. Since $\wh H$ is a Krull monoid, it follows that $\mathcal I_v^*(\wh H) = \mathcal I_v(\wh H) \setminus \{\emptyset\}$, and therefore we obtain a homomorphism \,$\vartheta = \vartheta' \t \mathcal I_v^*(H)\colon \mathcal I_v^*(H) \to \mathcal I_v^*(\wh H)$. Since $\vartheta(aH) = a\wh H$ for all $a \in H$, we obtain the following commutative diagram with exact rows (where $c$ is the homomorphism induced by the inclusion $H \hookrightarrow \wh H$, and $\partial_H \colon H_\red \to \mathcal I_v^*(H)$ is defined by $\partial(aH^\times) = aH$ for all $a \in H$).
\[
\begin{CD}
1 @> >>  \mathsf q(H)/H^\times @> \mathsf q( \partial_H) >> \mathsf q(\mathcal I_v^*(H)) @> >> \mathcal C_v(H) @>>> \boldsymbol 0 \\
@.          @V c VV              @V \mathsf q(\vartheta)  VV  @VV \overline \vartheta V @.\\
1 @> >>  \mathsf q(\wh H)/\wh H^\times @> \mathsf q( \partial_{\wh H})>>\mathsf q(\mathcal I_v^*(\wh H))  @>>> \mathcal C_v(\wh H)@>>> \boldsymbol 0\,.
\end{CD}
\]
Since $c$ is surjective and \,$\Ker(c) = \wh H^\times /H^\times$, we obtain an exact sequence
\[
1 \ \to \ \wh H^\times /H^\times \ \to \ \Ker (\mathsf q(\vartheta)) \ \to \ \Ker(\overline \vartheta) \ \to \boldsymbol 0,
\]
If $\mathfrak p \in \mathfrak X(H)$, then $S(\mathfrak p) =\{ \mathfrak P \in \mathfrak X(\wh H) \mid \mathfrak P \cap H = \mathfrak p\} = \{ \mathfrak q \cap \wh H \mid \mathfrak q \in \mathfrak X(\wh H_\mathfrak p)\}$, \ $\wh H_\mathfrak p = \wh{H_\mathfrak p}$ is factorial, and therefore
\[
(\wh H_\mathfrak p)_\red = \coprod_{\mathfrak q \in \mathfrak X(\wh H_\mathfrak p)} \bigl( (\wh H_\mathfrak p)_{\mathfrak q}\bigr)_\red = \coprod_{\mathfrak q \in \mathfrak X(\wh H_{\mathfrak p})} (\wh H_{\mathfrak q \cap \wh H})_\red = \coprod_{\mathfrak P \in S(\mathfrak p)} (\wh H_{\mathfrak P})_\red.
\]
Since $\mathsf q(H_\mathfrak p) = \mathsf q(\wh H_\mathfrak p) = \mathsf q(\wh H) = \mathsf q(H)$, we get
and therefore
\[
\coprod_{\mathfrak p \in \mathfrak X(H)} \mathsf q \bigl( (\wh H_\mathfrak p)_\red \bigr)=\coprod_{\mathfrak p \in \mathfrak X(H)}\mathsf q(H)/\wh H_\mathfrak p^\times=\coprod_{\mathfrak p \in \mathfrak X(H)} \coprod_{\mathfrak P \in S(\mathfrak p)}\mathsf q(H)/\wh H^\times_\mathfrak P=\coprod_{\mathfrak P \in \mathfrak X(\wh H)} \mathsf q(H)/\wh H^\times_\mathfrak P.
\]
For $\mathfrak p \in \mathfrak X(H)$, the inclusion \,$H_\mathfrak p \hookrightarrow \wh H_\mathfrak p$ yields an epimorphism
$\psi_\mathfrak p \colon \mathsf q(H)/H_\mathfrak p^\times \to \mathsf q(H)/\wh H_\mathfrak p^\times$. Hence we obtain an epimorphism
\[
\psi = (\psi_\mathfrak p)_{\mathfrak p \in \mathfrak X(H)} \colon \coprod_{\mathfrak p \in \mathfrak X(H)} \mathsf q(H)/H_\mathfrak p^\times \to \coprod_{\mathfrak p \in \mathfrak X(H)} \mathsf q(H)/\wh H_\mathfrak p^\times=\coprod_{\mathfrak P \in \mathfrak X(\wh H)} \mathsf q(H)/\wh H^\times_\mathfrak P
\]
and an obviously commutative diagram
\[
\begin{CD}
 \mathsf q(\mathcal I_v^*(H)) @> \mathsf q(\delta_H) >> \coprod_{\mathfrak p \in \mathfrak X(H)} \mathsf q(H)/H_\mathfrak p^\times \\
@V \mathsf q(\vartheta) VV              @V \psi  VV  \\
\mathsf q(\mathcal I_v^*(\wh H)) @> \mathsf q(\delta_{\wh H}) >> \coprod_{\mathfrak P \in \mathfrak X(\wh H)} \mathsf q(H)/\wh H_\mathfrak P^\times \,,
\end{CD}
\]
where $\mathsf q(\delta_H)$ and $\mathsf q(\delta_{\wh H})$ are the quotients of the isomorphisms given in Proposition 5.3 (for $H$ and $\wh H$). Hence $\mathsf q(\vartheta)$ is an epimorphism, and thus $\overline \vartheta \colon \mathcal C_v(H) \to \mathcal C_v(\wh H)$ is an epimorphism. Observe that $\mathsf q(\delta_H)$ induces an isomorphism
\[
\mathsf q(\delta_H) \t \Ker( \mathsf q(\vartheta)) \colon \Ker( \mathsf q(\vartheta)) \ \to \ \Ker(\psi) = \coprod_{\mathfrak p \in \mathfrak X(H)} \wh H_\mathfrak  p^\times/H_\mathfrak p^\times
\]
satisfying $\mathsf q(\delta_H)(aH^\times) = (aH_\mathfrak p^\times)_{\mathfrak p \in \mathfrak X(H)}$ for all $a \in \wh H^\times$, and therefore we obtain the exact sequence
\[
1 \ \to \ \wh H^\times /H^\times \ \to \ \coprod_{\mathfrak p \in \mathfrak X(H)} \wh H_\mathfrak p^\times /H_\mathfrak p^\times \ \overset \varepsilon \to \ \mathcal C_v(H) \ \overset{\overline \vartheta} \to \ \mathcal C_v(\wh H) \ \to \ \boldsymbol 0,
\]
where $\varepsilon$ is given as asserted.
\end{proof}

\medskip
\begin{theorem}[\bf{Algebraic structure}] \label{5.5}
Let $H$ be a seminormal $v$-noetherian weakly Krull monoid with $\emptyset \ne \mathfrak f = (H \DP \wh H) \subsetneq H$. We set $\mathcal P^* = \{ \mathfrak p \in \mathfrak X(H) \mid \mathfrak p \supset \mathfrak f \}$ and $\mathcal P = \mathfrak X (H) \setminus \mathcal P^*$.
\begin{enumerate}
\item $\wh H$ is Krull,   $\mathcal P^*$ is finite and, for each $\mathfrak p \in \mathcal P$,  $H_{\mathfrak p}$ is a discrete valuation monoid.
    For each $\mathfrak p \in \mathfrak X (H)$, $H_{\mathfrak p}$ is seminormal $v$-noetherian primary, and if
    $H$ is the multiplicative monoid of nonzero elements of a domain, then  $H_{\mathfrak p}$ is even finitely primary.

\smallskip
\item There is a monoid isomorphism
\[
\chi \colon \mathcal I_v^*(H) \to \mathcal F (\mathcal P) \times \prod_{\mathfrak p \in \mathcal P^*} (H_\mathfrak p)_\red \quad \text{satisfying} \quad \chi \t \mathcal  P = \id_{\mathcal P}.
\]
$\mathcal I_v^* (H)$ is seminormal $v$-noetherian weakly factorial, and the $v$-class group $\mathcal C_v (H)$ of $H$ is isomorphic to the weak divisor class group $\mathcal C (H)$ of $H$.

\smallskip
\item Suppose that $H_{\mathfrak p}$ is finitely primary   for each $\mathfrak p \in \mathfrak X (H)$.
    \begin{enumerate}
    \item There is an exact sequence
          \[
          1\rightarrow \wh{H}^\times/H^\times\rightarrow\coprod_{\mathfrak{p}\in
          \mathfrak{X}(H)}\wh{H}_{\mathfrak{p}}^\times/H^\times_{\mathfrak{p}}\overset{\ve}{\rightarrow}
          \mathcal{C}_v(H)\overset{\overline{\vartheta}}{\rightarrow}\mathcal{C}_v(\wh{H})\rightarrow 0   \,.
          \]

    \item Suppose that  $\wh H_{\mathfrak p}^{\times}/H_{\mathfrak p}^{\times}$ is finite for each $\mathfrak p \in \mathcal P^*$.
          Then $\mathcal I_v^* (H)$ is  a {\rm C}-monoid, and if   $\mathcal C_v (H)$ is finite, then $H$ is a {\rm C}-monoid.
    \end{enumerate}
\end{enumerate}
\end{theorem}

\begin{proof}
1.  Since $H$ is $v$-noetherian, the set of prime $v$-ideals containing a fixed element is finite, and thus $\mathcal P^*$ is finite. By \cite[Theorem 2.6.5]{Ge-HK06a}, $\wh H$ is Krull and, for each $\mathfrak p \in \mathcal P$, $H_{\mathfrak p}$ is a discrete valuation monoid. All $H_{\mathfrak p}$ are seminormal $v$-noetherian primary by Lemma \ref{5.1}. If $R$ is a domain and $H = R^{\bullet}$, then Lemma \ref{3.3} implies that $H_{\mathfrak p}$ is finitely primary for each $\mathfrak{p}\in\mathfrak{X}(H)$.

\smallskip
2. By 1., $(H_\mathfrak p)_\red \cong \N_0$ for all $\mathfrak p \in \mathcal P$, and thus Proposition \ref{5.3}.4 implies
\[
\mathcal I_v^*(H) \cong \coprod_{\mathfrak p \in \mathfrak X(H)} (H_\mathfrak p)_\red \cong \N_0^{(\mathcal P)} \time \coprod_{\mathfrak p \in \mathcal P^*} (H_\mathfrak p)_\red\cong \mathcal F(\mathcal P) \time \coprod_{\mathfrak p \in \mathcal P^*} (H_\mathfrak p)_\red.
\]
By Proposition \ref{5.3}.4, we further obtain that $\mathcal C_v(H) \cong \mathcal C(H)$.

\smallskip
3.(a) follows from Proposition \ref{exactsequence}, and it remains to show 3.(b). If $\mathfrak p \in \mathfrak X(H)$, then $H_\mathfrak p$ is a C-monoid by Lemma \ref{3.3}.1, and if $\wh {H_\mathfrak p}^\times /H_\mathfrak p^\times$ is finite for all $\mathfrak p \in \mathcal P^*$, then
\[
\mathcal F(\mathcal P) \time \coprod_{\mathfrak p \in \mathfrak X(H)} H_\mathfrak p \quad \text{and thus also} \quad \Bigl(\mathcal F(\mathcal P) \time \coprod_{\mathfrak p \in \mathfrak X(H)} H_\mathfrak p \Bigr)_\red = \mathcal F(\mathcal P) \time \coprod_{\mathfrak p \in \mathfrak X(H)} (H_\mathfrak p)_\red
\]
is a C-monoid by \cite[Theorem 2.9.16]{Ge-HK06a} (since obviously $\mathcal F(P)$ is a C-monoid). Hence $\mathcal I_v^*(H)$ is a C-monoid by 1. Since $H_\red$ is isomorphic to a saturated submonoid of $\mathcal I_v^*(H)$ with class group $\mathcal C_v(H)$, it follows that $H_\red$ and thus also $H$ is a C-monoid by \cite[Theorems 2.9.10 and 2.9.16]{Ge-HK06a}, provided that $\mathcal C_v(H)$ is finite.
\end{proof}

The equivalent conditions (a) - (e) of the next corollary play a central role in our main arithmetical result (Theorem \ref{5.6}) and in the characterization of half-factoriality (Theorem \ref{6.2}).

\medskip
\begin{corollary}\label{L:CharClsIso}
Let $H$ be a seminormal $v$-noetherian weakly Krull monoid with $\emptyset \ne \mathfrak f = (H \DP \wh H) \subsetneq H$ such that  $H_{\mathfrak{p}}$ is  finitely primary of rank one for each $\mathfrak{p}\in\mathfrak{X}(H)$. Let
\[
D=\coprod_{\mathfrak{p}\in\mathfrak{X}(H)}(H_{\mathfrak{p}})_{\red}\quad\text{ and let}\quad\pi\colon\mathsf q(D)\rightarrow\mathcal{C}(H)
\]
be the canonical homomorphism. For $\mathfrak{p}\in\mathfrak{X}(H)$ let
\[
\mu_{\mathfrak{p}}:\mathsf q(H)\rightarrow\mathsf q(D)\qquad\text{be defined by}\qquad\mu_{\mathfrak{p}}(x)(\mathfrak q)=\begin{cases} xH_{\mathfrak{q}}^{\times}&\text{if }\mathfrak{q}=\mathfrak{p}\\
H_{\mathfrak{q}}^{\times}&\text{if }\mathfrak{q}\not=\mathfrak{p}
\end{cases}
\text{ and let}
\]
\[
\psi:\mathsf q(H)/H^\times\rightarrow\mathsf q(D)\qquad\text{be defined by}\qquad\psi(xH^{\times})=(xH_{\mathfrak{q}}^{\times})_{\mathfrak{q}\in\mathfrak{X}(H)}.
\]
Then the following statements are equivalent{\rm \,:}
\begin{enumerate}
\item[(a)] The homomorphism $\overline \vartheta \colon \mathcal C_v(H) \to \mathcal C_v(\wh H)$ $($in the exact sequence of {\rm Prop. \ref{exactsequence})} is an isomorphism.

\smallskip
\item[(b)] $\pi \t \wh D^\times = 0$.

\smallskip
\item[(c)] $\wh D^\times =\psi(\wh H^\times /H^\times)$.

\smallskip
\item[(d)] $(\pi\circ\mu_{\mathfrak p})(u) = (\pi\circ\mu_{\mathfrak p})(v)$ for all $\mathfrak p \in \mathfrak X(H)$ and $u,\,v \in \mathcal A(H_\mathfrak p)$.

\smallskip
\item[(e)] $(uH_\mathfrak p \cap H)\mathsf q(H) = (vH_\mathfrak p \cap H)\mathsf q(H)$ for all $\mathfrak p \in \mathfrak X(H)$ and $u,\,v \in \mathcal A(H_\mathfrak p)$.
\end{enumerate}
\end{corollary}

\begin{proof}
In the exact sequence
\[
1 \ \to \ \wh H^\times /H^\times \ \overset {\psi_0}\to \ \coprod_{\mathfrak p \in \mathfrak X(H)} \wh H_\mathfrak p^\times /H_\mathfrak p^\times \ \overset \varepsilon \to \ \mathcal C_v(H) \ \overset{\overline \vartheta} \to \ \mathcal C_v(\wh H) \ \to \ \boldsymbol 0
\]
the homomorphism $\varepsilon$ is given by the following commutative diagram with exact rows, where the vertical maps $j_0$ and $j$ are inclusions and $\psi_0=\psi\t \wh H^\times /H^\times$:
\[
\begin{CD}
1 @> >>  \wh H^\times /H^\times @> \psi_0>> \wh D^\times @> \varepsilon >> \mathcal C_v(H)\\
@.          @V j_0 VV              @V j  VV  @VV \overline{\delta}_H V @.\\
1 @> >>  \mathsf q( H)/H^\times @> \psi>>\mathsf q(D)  @> \pi >> \mathcal C(H)@>>> \boldsymbol 0\,.
\end{CD}
\]
Hence the equivalence of (a), \,(b) and (c) follows since we have:
\[
\overline \vartheta \ \text{ is an isomorphism } \ \iff \ \varepsilon =0 \ \iff \ \psi(\wh H^\times /H^\times) = \wh D^\times \ \iff \ \pi \t \wh D^\times =0.
\]

\smallskip

(b) \,$\Leftrightarrow$\, (d)\, Note that $\pi \t \wh D^\times =0$ if and only if $(\pi\circ\mu_{\mathfrak p})\t\wh H_\mathfrak p^\times =0$ for all $\mathfrak p \in \mathfrak X(H)$. If $\mathfrak p \in \mathfrak X(H)$ and $q$ is a prime element of $\wh H_\mathfrak p$, then $\mathcal A(H_\mathfrak p) = \wh H_\mathfrak p^\times q$, and therefore $(\pi\circ\mu_{\mathfrak p})\t \wh H_\mathfrak p^\times =0$ if and only if $\pi\circ\mu_{\mathfrak p}$ is constant on $\mathcal A(H_\mathfrak p)$.

\smallskip

(d) \,$\Leftrightarrow$\, (e)\, Observe that $\mathsf q(\partial_H)(\mathsf q(H)/H^{\times})=\{aH\mid a\in\mathsf q(H)\}=\mathsf q(\mathcal{H})$. If $\mathfrak p \in \mathfrak X(H)$ and $u \in H_\mathfrak p$, then $(\overline\delta_H^{-1}\circ\pi\circ\mu_{\mathfrak p})(u) = (uH_\mathfrak p\cap H) \mathsf q(\mathcal H)$ by Proposition \ref{5.3}.3. Let $\mathfrak p\in\mathfrak X(H)$ and $u,v\in\mathcal A(H_\mathfrak p)$. Since $\overline\delta_H$ is an isomorphism, we infer that $(\pi\circ\mu_{\mathfrak p})(u) = (\pi\circ\mu_{\mathfrak p})(v)$ if and only if $(uH_\mathfrak p\cap H) \mathsf q(\mathcal H)=(vH_\mathfrak p\cap H) \mathsf q(\mathcal H)$ if and only if $(uH_\mathfrak p\cap H) \mathsf q(H)=(vH_\mathfrak p\cap H) \mathsf q(H)$.
\end{proof}

\medskip
\begin{examples} \label{5.7}~

1. (Noetherian Domains) A noetherian domain is weakly Krull if and only if every prime ideal of depth one has height one. This holds true for Cohen-Macaulay domains which include  all one-dimensional noetherian domains. A monoid domain $R[H]$ is seminormal if and only if $R$ and $H$ are seminormal \cite[Theorem 4.76]{Br-Gu09a}.
We refer to \cite[Chapter 4G]{Br-Gu09a}, \cite{Br-He98, Br-Li-Ro06a, Li12a}   for a detailed study of class groups and to investigations when  such monoid domains are Cohen-Macaulay.
For weakly Krull domains obtained by pullback constructions see \cite{An-Ch-Pa06a}.

Let  $R$  be a one-dimensional noetherian
domain such that its integral closure $\overline R$  is a finitely generated $R$-module.
Then its conductor $\mathfrak f = (R \DP \overline R)$ is nonzero (i.e.,  $R$  is an
order in the Dedekind domain  $\overline R$), the $v$-class group $\mathcal C_v (R)$ coincides with the Picard group $\Pic(R) $,
and the  exact sequence of Proposition \ref{exactsequence} has the form
\[
1 \to \overline R^\times /R^\times \to
(\overline R / \mathfrak f)^\times \big/ (R/\mathfrak f)^\times \to
\text{Pic}(R) \to \mathcal \Pic (\overline R) \to 1 \ .
\]
Hence, if  $\overline R $  has finite class group
and all proper residue class rings are finite, then  $\Pic (R) $   is finite. This holds true for orders in algebraic number fields and for orders in holomorphy rings of algebraic function fields. In these cases every class $g \in \Pic (R)$ contains infinitely many prime ideals.

\smallskip
2. (Congruence monoids in Dedekind domains) Let $R$ be a Dedekind domain, $\{0\} \ne \mathfrak f \ne R$ an ideal such that $R/\mathfrak f$ is finite,
$\sigma$ a sign vector of $R$, $\emptyset\ne\Gamma\subset
R/\mathfrak f \sigma$ a multiplicatively closed subset and
$H_\Gamma=\{a\in R^{\bullet}\mid [a]_{\mathfrak{f}\sigma}\in\Gamma\}\cup\{1\}\subset R^\bullet$ \ the congruence monoid defined in
$R$ modulo $\mathfrak{f}\sigma$ by $\Gamma$. By \cite[Theorem 2.11.11]{Ge-HK06a}, $H_{\Gamma}$ is a $v$-noetherian monoid with $(H_{\Gamma}\DP\wh{H_{\Gamma}})\ne\emptyset$. In many situations, $H_{\Gamma}$ is a weakly Krull monoid. We provide some simple examples. If $\mathfrak o$ is an order in $R$, then $\mathfrak o^{\bullet}$ is a congruence monoid in $R$, and being a one-dimensional noetherian domain, it is weakly Krull by 1. If $H$ is a congruence monoid in $R$,   $\mathfrak f$  the largest ideal of definition for $H$, and  $aR + \mathfrak f = R$ for all $a \in H$, then $H$ is Krull (\cite[Proposition 2.11.6]{Ge-HK06a}).   Explicit arithmetical results have been derived for non-regular congruence monoids in the integers (\cite{B-C-C-M07, Ba-Ch-Sc08a, B-C-C-M09,Ch-St10a, Ba-Ch13a}).

However, $H_{\Gamma}$ is not always a weakly Krull monoid (not even if it is seminormal). To provide an example,
consider the monoid $H = p\Z^{\bullet} \cup \{1\}$ for some prime $p \in \Z$. Obviously, $H$ is seminormal and $\widehat H = \Z^{\bullet}$. For every prime $q \in \Z$, we have $q \Z^{\bullet} \cap H \in v$-$\spec (H)$, and if $q \ne p$, then $q\Z^{\bullet} \cap H \subsetneq H \setminus \{1\} = p \Z^{\bullet} \cap H$. Thus $H$ is not weakly Krull since $v$-$\max(H)\not=\mathfrak{X}(H)$.

\smallskip
3. (Weakly factorial monoids) As already mentioned, every coproduct of a family of primary monoids is weakly factorial. By \cite[Exercise 22.5]{HK98}, a (not necessarily $v$-noetherian) weakly Krull monoid is weakly factorial if and only if it has trivial $t$-class group. In particular, a one-dimensional domain is weakly factorial if and only if it has trivial Picard group.  Generalized Cohen-Kaplansky domains and  weak Cohen-Kaplansky domains are weakly factorial  (\cite[Theorem 4]{An-An-Za92b}, \cite[Proposition 8]{Ch-An12a}, and \cite{Co-Sp12a}). Weakly factorial semigroup rings were studied by Chang \cite{Ch09a}. For characterizations and  generalizations of weakly factorial domains (which are all weakly Krull) see \cite{An-Za90, An-Ch-Pa03a}.

\smallskip
4. (Value semigroups) Value semigroups of one-dimensional semilocal rings reflect ring-theoretical properties (such as being Gorenstein or being Arf). In \cite{Ba-An-Fr00a}, Barucci, D'Anna, and Fr{\"o}berg study so-called good semigroups as purely semigroup theoretical models for value semigroups. Each good semigroup is a finite direct product of good local semigroups (\cite[Theorem 2.5]{Ba-An-Fr00a}), and a good local semigroup is a finitely  primary monoid with complete integral closure being equal to $(\N_0^s, +)$ for some $s \ge 1$. Thus all good semigroups are weakly factorial.

\smallskip
5. (Ideal semigroups) If $H$ is a $v$-noetherian weakly Krull monoid, then its monoid $\mathcal I_v^* (H)$ of $v$-invertible $v$-ideals is $v$-noetherian weakly factorial (see Theorem \ref{5.5}.2).

\smallskip
6. (Seminormality) The seminormalization $H'$ and the root closure $\widetilde H$ of any weakly Krull monoid $H$ are seminormal weakly Krull monoids, and if $\wh H$ is a Krull monoid, then $H'$ and $\widetilde H$ are both $v$-noetherian (\cite[Theorem 4.1]{Re13a}).
\end{examples}

\smallskip
7. (Seminormal quadratic orders) Seminormal orders in algebraic number fields have been studied by Dobbs and Fontana \cite{Do-Fo87}, who (among others) gave a full characterization of seminormal orders in quadratic number fields.

Let $\Delta \in \Z$ be a quadratic fundamental discriminant (that means, either $\Delta \equiv 1 \mmod 4$ is squarefree, or $\Delta = 4D$ for some squarefree $D \in \Z$ such that $D \not \equiv 1 \mmod 4$). Let $\sigma \in \{0,1\}$ be such that $\sigma \equiv \Delta \mmod 2$, and set
\[
\omega_\Delta = \frac{\sigma + \sqrt \Delta}2.
\]
Then $\Z[\omega_\Delta]$ is the ring of integers of the quadratic number field $\Q( \sqrt \Delta\,)$. If $f \in \N$, then $\Z[f\omega_\Delta]$ is the only order of index $f$ in $\Z[\omega_\Delta]$,
\[
\bigl|{\rm Pic}\bigl(\Z[f\omega_\Delta]\bigr)\bigr| = \bigl|{\rm Pic}\bigl(\Z[\omega_\Delta]\bigr)\bigr|\,\frac f {(\Z[\omega_\Delta]^\times \DP \Z[f\omega_\Delta]^\times)}\,\prod_{p \t f} \Bigl[ 1 - \frac 1 p \Bigl(\frac \Delta p \Bigr)\Bigr],
\]
and every class of \,${\rm Pic}(R) = \mathcal C_v(R^\bullet)$ contains infinitely many prime ideals (see \cite[Theorems 5.1.7, 5.9.7 and 8.2.7]{HK13a}). By \cite[Corollary 4.5]{Do-Fo87}, \,$\Z[f \omega_\Delta]$ is seminormal if and only if $f$ is squarefree and $\gcd (\Delta ,f)=1$. In this case,
\[
\frac{\bigl|{\rm Pic}\bigl(\Z[f\omega_\Delta]\bigr)\bigr|}{\bigl|{\rm Pic}\bigl(\Z[\omega_\Delta]\bigr)\bigr|} \, = \,\frac 1 {(\Z[\omega_\Delta]^\times \DP \Z[f\omega_\Delta]^\times)}\,\prod_{p \t f} \Bigl[ p - \Bigl(\frac \Delta p \Bigr)\Bigr],
\]
In particular, if $\Delta \equiv 1 \mmod 4$, then $\Z[\sqrt \Delta\,] = \Z[2\omega_\Delta]$ is seminormal, and the homomorphism \newline $\overline \vartheta \colon {\rm Pic}(\Z[\sqrt \Delta\,]) \to  {\rm Pic}(\Z[\omega_\Delta])$ of Theorem \ref{5.5} is an isomorphism if and only if either $\Delta \equiv 1 \mmod 8$, or $\Delta \equiv 5 \mmod 8$ and \,$\Z[\sqrt \Delta\,]^\times \subsetneq \Z[\omega_\Delta]^\times$. In the latter case, we have \,$(Z[\omega_\Delta]^\times \DP \Z[\sqrt \Delta\,]^\times) = 3$\, (see \cite[Theorem 5.2.3]{HK13a}).

\medskip
Our main arithmetical result, Theorem \ref{5.6}, considers unions of sets of lengths and (monotone) catenary degrees. Moreover, it shows the existence of a certain transfer homomorphism. In order to discuss its significance, let us fix an atomic monoid $H$. By the basic inequality (\ref{basic3}), we have $2 + \sup \Delta (H) \le \mathsf c (H)$. In particular, $\mathsf c (H) = 2$ implies that $H$ is half-factorial, and $\mathsf c (H) = 3$ implies that all sets of lengths are intervals.

Unions of sets of lengths were introduced by Chapman and Smith in \cite{Ch-Sm90a}. In particular, their suprema -- these are the $\rho_k (H)$-invariants  -- and the elasticity --- this is the limit of the $\rho_k (H)/k$ --- have received much attention since the early days of factorization theory.
Under very mild finiteness conditions, there exist constants $M$ and $k^*$ such that, for all $k \ge k^*$, all unions $\mathcal U_k (H)$ are almost arithmetical progressions with  bound $M$ (\cite[Theorem 4.2]{Ga-Ge09b}).
Much research has been motivated by the question (\cite[Problem 37]{Ch-Gl00a}) of whether, in case of Krull monoids with finite class group where every class contains a prime divisor, unions of sets of lengths are intervals. This question was answered in the affirmative in \cite{Fr-Ge08}. For weakly Krull monoids (or, more generally for {\rm C}-monoids) there are simple characterizations which guarantee that the unions are finite, but their precise structure has not been determined yet.

The first finiteness result for monotone catenary degrees is due to Foroutan where it was established for certain noetherian weakly Krull domains (\cite[Theorem 5.2]{Fo06a}). The finiteness of the monotone catenary degree seems to be a rare phenomenon (inside the class of objects having finite catenary degree), and explicit results are very sparse, even for Krull monoids with finite class groups (see \cite{Fo-Ge05, Fo-Ha06b, Ge-Gr-Sc-Sc10, Ge-Yu13a}). In \ref{5.12}, we provide an example showing that  the monotone catenary degree can be infinite when we drop the seminormality condition.

Under the assumptions of Theorem \ref{5.6}.2.(a).(ii), there is a transfer homomorphism  $H \to \mathcal B (G)$. Note that $\mathcal B (G)$ is a commutative Krull monoid with class group isomorphic to $G$, and it is well-known that there is a transfer homomorphism  $\widehat H \to \mathcal B (G)$. But it is surprising to have  a transfer homomorphism  from a non-Krull monoid $H$ to a Krull monoid (very recently, similarly surprising transfer homomorphisms were established from certain non-commutative Krull monoids onto  monoids of zero-sum sequences over  finite abelian groups (\cite{Sm13a})). The existence of this transfer homomorphism implies that all of the precise arithmetical results that have been established for $\mathcal B (G)$ (see \cite[Chapter 6]{Ge-HK06a} and \cite{Ge09a} for surveys) hold also for $H$, and that all arithmetical invariants coincide with those for $\widehat H$.

\medskip
\begin{theorem}[\bf{Arithmetic structure}] \label{5.6}
Let $H$ be a seminormal $v$-noetherian weakly Krull monoid with $\emptyset \ne \mathfrak f = (H \DP \wh H) \subsetneq H$ such that $H_{\mathfrak p}$ is finitely primary of rank $s_{\mathfrak p}$ for each $\mathfrak{p}\in\mathfrak{X}(H)$. Let $\pi \colon \mathfrak X(\wh H) \to \mathfrak X(H)$ be the natural map defined by \,$\pi(\mathfrak P) = \mathfrak P \cap H$ for all $\mathfrak P \in \mathfrak X(\wh H)$, and let \,$\overline \vartheta \colon \mathcal C_v(H) \to \mathcal C_v(\wh H)$ be the epimorphism given by {\rm Proposition \ref{exactsequence}}.
\begin{enumerate}
\item
      \begin{enumerate}
      \smallskip
      \item If  $\pi$ is bijective, then $\mathsf c \big(\mathcal I_v^* (H) \big) = 2$.

      \smallskip
      \item Suppose  $\pi$ is not bijective.
            Then
            $\mathcal U_k \big( \mathcal I_v^* (H) \big) = \N_{\ge 2}$ for all $k \ge 2$,   $\Delta \big(
            \mathcal I_v^* (H) \big) = \{1\}$, and $\mathsf c \big( \mathcal I_v^* (H) \big) = \mathsf c_{\adj} \big( \mathcal I_v^* (H) \big) = 3$.
            If there is precisely one  $\mathfrak p \in \mathfrak X (H)$ with $s_{\mathfrak p} > 1$, then
            $\mathsf c_{\monn} \big( \mathcal I_v^* (H) \big) =3$, and otherwise $\mathsf c_{\monn} \big( \mathcal I_v^* (H) \big)
            =5$.
      \end{enumerate}

\smallskip
\item Suppose that  $G = \mathcal C_v (H)$ is finite, and that every
      class contains a  $\mathfrak p \in \mathfrak X (H)$ with $\mathfrak p \not\supset \mathfrak f$.
      \begin{enumerate}
      \smallskip
      \item Suppose that $\pi$ is bijective.
            \begin{enumerate}
            \smallskip
            \item $\mathcal U_k (H)$ is a finite interval for all $k \ge 2$.

            \smallskip
            \item Suppose that $\overline{\vartheta}\colon \mathcal{C}_v(H)\rightarrow \mathcal{C}_v(\wh{H})$ is an isomorphism.
                  Then  there is a transfer homomorphism $\theta \colon H \to \mathcal B ( G )$ such that $\mathsf{c}(H,\theta)\leq 2$. Thus
                  $\mathcal L (H) = \mathcal L \big( \mathcal B(G) \big)$, and if $|G| \ge 3$, then $\mathsf c (H) = \mathsf c \big( \mathcal B(G) \big)$ and $\mathsf c_{\monn} (H) = \mathsf c_{\monn} \big( \mathcal B(G) \big)$.
            \end{enumerate}

      \smallskip
      \item Suppose  that $\pi$ is not bijective. Then,
            for all $\ell \ge 2$ and $k\in [2,3]$, we have $\ell \in \mathcal U_k (H)$ or $\ell+1 \in \mathcal U_k (H)$ and, for all $k\geq 4$, we have $\N_{\ge 4} \subset \mathcal U_k (H)\subset\N_{\ge 2}$.
      \end{enumerate}
\end{enumerate}
\end{theorem}

\begin{proof}
We set $\mathcal P^* = \{ \mathfrak p \in \mathfrak X(H) \mid \mathfrak p \supset \mathfrak f \}$ and $\mathcal P = \mathfrak X (H) \setminus \mathcal P^*$.
By Theorem \ref{5.5}, there exists an isomorphism \,$\chi \colon \mathcal I_v^*(H) \to D = \mathcal F(\mathcal P) \time D_1 \time \ldots \time D_n$, where $n = |\mathcal P^*| \in \N$, \ $D_1, \ldots, D_n$ are reduced seminormal finitely primary monoids, and $\chi \t \mathcal P = \id_{\mathcal P}$. Hence \,$\chi \circ \partial_H \colon H_\red \to \mathcal H = \{aH
\mid a \in H \} \hookrightarrow \mathcal I_v^*(H) \to D$\, induces an isomorphism \,$H_\red \to H^*$ onto a cofinal and saturated submonoid $H^*$ of $D$, and there is a natural isomorphism \,$\overline \chi \colon \mathcal C_v(H)= \mathcal C(\mathcal H, \mathcal I_v^*(H)) \to G = \mathcal C(H^*,D)$ mapping classes of primes onto classes of primes (use Lemma \ref{4.0}). Thus we may assume from now on that $H = H^* \subset D$ is a cofinal saturated submonoid with class group $G = \mathcal C(H,D)$. By Lemma \ref{5.1}.3, $\pi$ is bijective if and only if $D_1, \ldots, D_n$ are all of rank $1$.

\smallskip
1.  All  assertions follow from the structural description of $\mathcal I_v^* (H)$, from Lemma \ref{3.5}, and from Theorem \ref{3.7}.

\smallskip
2. By Lemma \ref{4.2}, it suffices to prove the assertions 2.(a)(i) and 2.(b) for the monoid $B = \mathcal B (G, T, \iota) \subset F = \mathcal F (G) \time T$, where $T= D_1 \time \ldots \time D_n$. Since $G$ is finite, $H \subset D$ is cofinal, and hence $B \subset F$ is saturated, cofinal, and $G=\mathcal C (B,F)$. In particular, for each $a \in F$ we have $a \in B$ if and only if $[a] = 0 \in G$ and hence $a^{\exp (G)} \in B$. We will often use Lemma \ref{3.3}.1.

\smallskip
2.(a)(i) For $i \in [1,n]$, we  set $D_i \subset \wh{D_i} = \wh{D_i}^{\times} \time [q_i]$, and by Lemma \ref{3.5} we have $\mathcal A (D_i) = \{\epsilon q_i \mid \epsilon \in \wh{D_i}^{\times} \}$. Every $S \in F$ has a unique product decomposition of the form
\[
S = g_1 \cdot \ldots \cdot g_k a_1 \cdot \ldots \cdot a_n \,,
\]
where $k \in \N_0$, $g_1, \ldots, g_k \in G$, and $a_i \in D_i$ for all $i \in [1,n]$, and
\[
\mathsf Z_F (S) = g_1 \cdot \ldots \cdot g_k \prod_{i=1}^n \mathsf Z_F (a_i) \,.
\]
Furthermore,
\[
\mathcal A (F) = G \cup \bigcup_{i=1}^n \mathcal A (D_i) \,,
\]
and since $D_1, \ldots, D_n$ have catenary degree at most two by Lemma \ref{3.5}.2,  $F$ is half-factorial by Theorem \ref{3.7}.1. All this information will be used in our divisibility arguments below without further mention.

Let  $S = u_1 \cdot \ldots \cdot u_{\ell} \in F$, where $\ell \in \N_0$ and $u_1, \ldots, u_{\ell} \in \mathcal A (F)$. Then   $\mathsf L_F (S) = \{\ell\}$, and (for the rest of this proof) we set $|S| = \ell$.  We continue with the following assertion.

\begin{enumerate}
\item[{\bf A1.}\,] Let $S = u_1 \cdot \ldots \cdot u_{\ell} \in \mathcal A (B)$, where $\ell \in \N_0$ and $u_1, \ldots, u_{\ell} \in \mathcal A (F)$,  let $k \in [1, \ell]$,  and $g \in G$ with $g = [u_1 \cdot \ldots \cdot u_k] \in G$. Then $S' = g u_{k+1} \cdot \ldots \cdot u_{\ell} \in \mathcal A (B)$ and $|S'| = |S|-k+1$.

\end{enumerate}

\smallskip

{\it Proof of \,{\bf A1}}.\, Since $[S'] = [S] = 0 \in G$, it follows that $S' \in B$, and by definition we have  $|S'| = |S|-k+1$.
Assume to the contrary that $S'$ is not an atom of $B$, whence $S' = g u_{k+1} \cdot \ldots \cdot u_{\ell} = S_1 S_2$ where $S_1, S_2 \in B \setminus \{1\}$. Without restriction we may suppose that $g \mid_F S_1$. This implies that $S_2 \mid_F u_{k+1} \cdot \ldots \cdot u_{\ell} \mid_F S$, and hence $S_2 \mid_B S$, a contradiction. \qed ({\bf A1)}

\smallskip
Since $G$ is finite, it follows that the Davenport constant $\mathsf D (G)$ of $G$ is finite (see \cite[Section 5.1]{Ge-HK06a}). Then  \cite[Proposition 3.4.5]{Ge-HK06a} implies that  $\sup \mathsf L_F (u) \le \mathsf D (G)$ for all $u \in \mathcal A (B)$ and that $\mathsf \rho_k (B) \le \rho_{k \mathsf D (G)} (F) = k \mathsf D (G) < \infty$ for all $k \in \N$. Thus all sets $\mathcal U_k (B)$ are finite.

First we assert that it is sufficient to prove that $[k, \rho_k (B)] \subset \mathcal U_k (B)$ for each $k \in \N$. Suppose  this holds and let $k \in \N$.  If $\ell \in [\lambda_k (B), k]$, then $k \in [\ell, \rho_{\ell} (B)]$, hence $k \in \mathcal U_{\ell} (B)$, and thus $\ell \in \mathcal U_k (B)$.

Now let $k \in \N$, and let $\ell \in [k, \rho_k (B)]$ be minimal such that $[\ell, \rho_k (B)] \subset \mathcal U_k (B)$. We have to show that $\ell = k$. Assume to the contrary that $\ell > k$.
We denote by $\Omega$ the set of all $A \in B$ such that $\{k, j\} \subset \mathsf L_B (A)$ for some $j \ge \ell$, and continue with the following assertion.

\begin{enumerate}
\item[{\bf A2.}\,] $\Omega \cap \mathcal B (G) \ne \emptyset$.
\end{enumerate}

\smallskip
{\it Proof of} \,{\bf A2}.\, Let $A \in  \Omega \setminus \mathcal B (G)$ and $U_1, \ldots, U_k, V_1, \ldots, V_j \in \mathcal A (B)$ such that $A = U_1 \cdot \ldots \cdot U_k = V_1 \cdot \ldots \cdot V_j$. We set $A = A_1A_2$, where $A_1 \in \mathcal F (G)$ and $A_2 \in T$. We show that there is an $A' = A_1'A_2' \in \Omega$, where $A_1' \in \mathcal F (G)$, $A_2' \in T$, such that $|A_2'| < |A_2|$. Then the assertion follows by an inductive argument. We start with a first reduction step.

\noindent
{\it Reduction Step.} Suppose there are $i \in [1,n]$, $u \in \mathcal A (D_i)$, $\lambda \in [1,k]$, and $\mu \in [1,j]$ such that $u \mid_F U_{\lambda}$ and $u \mid_F V_{\mu}$.

After renumbering if necessary, we may suppose that $\lambda = \mu = 1$. We set $g = [u] \in G$, $U_1' = g u^{-1} U_1$, $V_1' = g u^{-1} V_1$, and $A' = g u^{-1} A$. Then $U_1', V_1' \in \mathcal A (B)$ by ({\bf A1}), and $A' = U_1' U_2 \cdot \ldots \cdot U_k = V_1'V_2 \cdot \ldots \cdot V_j \in \Omega$ has the required property.

 Since $A \notin \mathcal B (G)$,  there is an $i \in [1,n]$, such that $q_i \mid_{\wh F} A$, say $i =1$. If $q_1^2 \nmid_{\wh F} A$, then there are $\epsilon \in \wh{D_1}^{\times}$, $\lambda \in [1,k]$, and $\mu \in [1,j]$ such that $(\epsilon q_1) \mid_F U_{\lambda}$ and $(\epsilon q_1) \mid_F V_{\mu}$. Since $\epsilon q_1 \in \mathcal A (D_1)$, the assumption of the {\it Reduction Step} is satisfied, and we are done. From now on we suppose that $q_1^2 \mid_{\wh F} A$, and we distinguish two cases.

\smallskip
\noindent CASE 1: \ There is a $\lambda \in [1,k]$ such that $q_1^2 \mid_{\wh F} U_{\lambda}$ or there is a $\mu \in [1,j]$ such that $q_1^2 \mid_{\wh F} V_{\mu}$.

First suppose that $q_1^2 \mid_{\wh F} U_{\lambda}$. Then there is a $\mu \in [1,j]$ such that $(\epsilon q_1) \mid_F V_{\mu}$ for some $\epsilon \in \wh{D_1}^{\times}$. Since $u = \epsilon q_1 \in \mathcal A (D_1)$ and $u \mid_F U_{\lambda}$, the assumption of the {\it Reduction Step} is satisfied, and we are done. If $q_1^2 \mid_{\wh F} V_{\mu}$, then the argument is similar.

\smallskip
\noindent CASE 2: For all $W \in \{U_1, \ldots, U_k, V_1, \ldots, V_j\}$, we have $q_1^2 \nmid_{\wh F} W$.

Since $q_1^2 \mid_{\wh F} A$, we may suppose, after renumbering if necessary, that $(\epsilon_1 q_1) \mid_F V_1$, $(\epsilon_2 q_1) \mid_F V_2$, and $(\eta q_1) \mid_F U_1$ where $\epsilon_1, \epsilon_2, \eta \in \wh{D_1}^{\times}$.
We set $g = [\eta q_1] \in G$. Then $V_1' = (\eta q_1)^{-1}g V_1 V_2 \in B$, $U_1' = (\eta q_1)^{-1} g U_1 \in \mathcal A (B)$ by ({\bf A1}), and $U_1' U_2 \cdot \ldots \cdot U_k = V_1'V_3 \cdot \ldots \cdot V_j$.
If $V_1' \in \mathcal A (B)$ and $j-1 = \ell-1$, then $\ell-1 \in \mathcal U_k (B)$, a contradiction. Otherwise $V_1'V_3 \cdot \ldots \cdot V_j$ gives rise to a factorization of length greater than or equal to $\ell$ and hence $A' = U_1'U_2 \cdot \ldots \cdot U_k \in \Omega$. If $A' = A_1'A_2'$ with $A_1' \in \mathcal F (G)$ and $A_2' \in T$, then $|A_2'| < |A_2|$, and we are done. \qed ({\bf A2)}

\smallskip
Now we choose some $A \in \Omega \cap \mathcal B (G)$ such that $|A|$ is minimal.
Then there are $U_1, \ldots, U_k, V_1, \ldots, V_j \in \mathcal A (B)$ such that $A = U_1 \cdot \ldots \cdot U_k = V_1 \cdot \ldots \cdot V_j$.

Assume to the contrary that $|U_1| = \ldots = |U_k| = 1$. Then $U_1, \ldots, U_k$ are atoms in $F$. Since $F$ is half-factorial,   we infer that
\[
j \le \max \mathsf L_B (A) \le \max \mathsf L_F (A) = k < \ell \,,
\]
a contradiction.

Therefore, there exists a $\lambda \in [1,k]$ such that $|U_{\lambda}| \ge 2$, say $U_k = g h U'$ where $g, h \in G$ and $U' \in F$.
After renumbering if necessary we obtain that $V_{j-1}V_j = g h V'$ for some $V' \in F$. We set  $U_k' = (g+h) U'$ and $V_{j-1}' = (g+h) V'$. Then $U_k' \in \mathcal A (B)$ and $V_{j-1}' \in \mathcal B (G)$, say $V_{j-1}' = W_1 \cdot \ldots \cdot W_t$ where $t \in \N$ and $W_1, \ldots, W_t \in \mathcal A (G)$. If $A' = U_1 \cdot \ldots \cdot U_{k-1} U_k'$, then $|A'| < |A|$ and $A' = V_1 \cdot \ldots \cdot V_{j-2}W_1 \cdot \ldots \cdot W_t$. By the minimality of $|A|$ it follows that $j-2+t < \ell$, hence $t=1$, $j = \ell$, and $\ell-1 \in \mathcal U_k (B)$, a contradiction.

\smallskip
2.(a)(ii) Since $H\subset D$ is saturated, $\wh{H}\subset \wh{D}=\mathcal{F}(\mathcal P)\times
\wh{D}_1\times\ldots \times \wh{D}_n$ is saturated too (\cite[Lemma 3.3.1]{Ge-Ha08b}). From $H\subset
D\cap\wh{H}\subset D\cap \mathsf{q}(H)=H$ we obtain $H=D\cap \wh{H}$. From Corollary \ref{L:CharClsIso}
we obtain $\wh{D}^\times=\wh{H}^\times D^\times$. Applying Proposition \ref{L:GlobalTransfer} we
conclude that the inclusion $\iota\colon H\rightarrow \wh{H}$ is a transfer homomorphism with
$\mathsf{c}(H,\iota)\leq 2$.

We assert that every class of $\mathcal C_v(\wh H)$ contains some $\mathfrak P \in \mathfrak X(\wh H)$, and in order to verify this we use the isomorphism $\overline{\vartheta} \colon \mathcal{C}_v(H)\rightarrow \mathcal{C}_v(\wh{H})$.  If $g \in \mathcal C_v(\wh H)$, then there exists some $\mathfrak p \in \mathcal P$ such that $g = \overline \vartheta([\mathfrak p])$, and then $H_{\mathfrak p}$ is a discrete valuation monoid by Theorem \ref{5.5}.1. Hence $H_{\mathfrak p} = \wh{H_{\mathfrak p}} = \wh H_{\mathfrak p}$, and thus $\mathfrak P = \mathfrak p H_\mathfrak p \cap \wh H \in \mathfrak X(\wh H)$. Since $\mathfrak p \subset\mathfrak P$, we obtain $\mathfrak P = \mathfrak p_{v(\wh H)}$, and thus $\mathfrak P \in \overline \vartheta ([\mathfrak p]) = g$.

Therefore, by Lemma \ref{4.2} (applied with $T = \{1\}$ and $G_P = \mathcal C_v (\widehat H)$), there is a transfer homomorphism
 $\boldsymbol \beta \colon \wh{H}\rightarrow
\mathcal{B}(\mathcal{C}_v(\wh{H})) \cong \mathcal{B}(G)$ with $\mathsf{c}(\wh{H}, \boldsymbol \beta)\leq 2$. Hence, by Lemma \ref{4.1},
$\theta = \boldsymbol \beta\circ \iota \colon H \to \mathcal B (\mathcal C_v (H) )$ is a transfer homomorphism satisfying $\mathsf c (H, \theta) \le 2$. Now all remaining assertions follow (note that $|G|\ge 3$, implies that $\mathcal B (G)$ is not factorial whence $\mathsf c \big( \mathcal B (G) \big) \ge 3$).

\smallskip
2.(b)  We may suppose that $D_1$ has rank $s \ge 2$, say
\[
D_1 \subset \wh{D_1} = \wh{D_1}^{\times} \time [q_1,
\ldots, q_s] \ , \quad  D_1 \setminus D_1^\times \subset q_1 \cdot
\ldots \cdot q_s \wh{D_1} \quad \text{and} \quad (q_1 \cdot
\ldots \cdot q_s) \wh{D_1} \subset D_1 \,,
\]
where $q_1, \ldots, q_s$ are pairwise non-associated prime elements
of $\wh{D_1}$. If $|G| = 1$, then $B = F$, and the assertion follows from Theorem \ref{3.7}. From now on we suppose that $|G| \ge 2$, and we start with the following assertions.

\begin{enumerate}
\item[{\bf A3.}\,] Let $1 \ne a \in D_1 \cap B$. Then $2 \in \mathsf L_B (a)$ or $2 \in \mathsf L_B (au)$ for some $u \in \mathcal A (B)$.

\smallskip
\item[{\bf A4.}\,]    There is an element $c \in B$ such that $\{2,3\} \subset \mathsf L_B (c)$.
\end{enumerate}

\smallskip
{\it Proof of} \,{\bf A3}.\, Suppose that  $a = \epsilon q_1^{m_1} \cdot \ldots \cdot q_s^{m_s}$ with $\epsilon \in \wh{D_1}^{\times}$ and $m_1, \ldots, m_s \in \N$. If $a \in \mathcal A (B)$, then  $2 \in \mathsf L_B (au)$ for each $u \in \mathcal A (B)$. Suppose that  $a \notin \mathcal A (B)$. Then $a \notin \mathcal A (D_1)$ and hence $\min \{m_1, \ldots, m_s\} \ge 2$. Consider the following product decomposition
\[
a = \big( \epsilon q_1 q_2^{m_2-1} \cdot \ldots \cdot q_s^{m_s-1} \big) \big( q_1^{m_1-1} q_2 \cdot \ldots \cdot q_s \big) \,.
\]
If one of the two factors lies in $B$, then so does the other one, and then both elements are atoms of $B$. Thus $2 \in \mathsf L_B (a)$. Otherwise, suppose that $[ \epsilon q_1 q_2^{m_2-1} \cdot \ldots \cdot q_s^{m_s-1}] = g \in G \setminus \{0\}$. Then $(-g) \epsilon q_1 q_2^{m_2-1} \cdot \ldots \cdot q_s^{m_s-1} \in \mathcal A (B)$, $g q_1^{m_1-1} q_2 \cdot \ldots \cdot q_s \in \mathcal A (B)$, and their product is equal to $\big( (-g)g \big) a$. Thus $2 \in \mathsf L_B \big( (-g)ga \big)$. \qed ({\bf A3)}

\smallskip
{\it Proof of} \,{\bf A4}. \, If $|G| \ge 3$, then obviously there is a $c \in \mathcal B (G) \subset B$ such that $\mathsf L_B (c) = \{2,3\}$ (see, for example, \cite[Theorem 6.6.2]{Ge-HK06a}). Suppose that $|G| = 2$, set $G = \{0, g\}$, $u = q_1 \cdot \ldots \cdot q_s$, and distinguish the following two cases.

\noindent
CASE 1: \ $u \in B$.

We consider the equation
\[
u^3 = (q_1q_2^2 \cdot \ldots \cdot q_s^2)(q_1^2 q_2 \cdot \ldots
\cdot q_s) \,.
\]
If $ q_1 q_2^2 \cdot \ldots \cdot q_s^2 \in B$, then $q_1^2 q_2
\cdot \ldots \cdot q_s \in B$, both elements are atoms of $B$, and
$u^3$ can be written as a product of two atoms. Suppose that $ q_1
q_2^2 \cdot \ldots \cdot q_s^2 \notin B$. Then $q_1^2 q_2 \cdot
\ldots \cdot q_s \notin B$, and hence $ q_1 q_2^2 \cdot \ldots \cdot
q_s^2 g$ and  $q_1^2 q_2 \cdot \ldots \cdot q_sg$ are atoms in $B$.
The equation
\[
(q_1 \cdot \ldots \cdot q_s) ( q_1^3q_2 \cdot \ldots \cdot q_s) =
(q_1^2q_2 \cdot \ldots \cdot q_s)^2 \in B
\]
shows that $q_1^3q_2 \cdot \ldots \cdot q_s \in B$ and hence it is
an atom of $B$. Finally the equation
\[
(q_1q_2^2 \cdot \ldots \cdot q_s^2g)(q_1^3q_2 \cdot \ldots \cdot
q_s) = (q_1 \cdot \ldots \cdot q_s)^2 (q_1^2q_2 \cdot \ldots \cdot
q_sg)
\]
shows that a product of two atoms of $B$ can be written as a product
of three atoms of $B$.

\noindent
CASE 2: \ $u \notin B$.

Then $v = u^2 \in \mathcal A (B)$, and $|\{ i \in [1,s] \mid [q_i] = g \}|$ is odd, say $[q_1] = \ldots = [q_{2k+1}] = g$ and $[q_{2k+2}] = \ldots = [q_s] = 0$ for some $k \in [0, (s-1)/2]$. If $k = 0$, then
\[
c = v (q_1^4q_2 \cdot \ldots \cdot q_s) = (q_1^2 q_2 \cdot \ldots \cdot q_s)^3
\]
has the required property, since $q_1^4q_2 \cdot \ldots \cdot q_s, q_1^2 q_2 \cdot \ldots \cdot q_s \in \mathcal A (B)$. If $k \ge 1$, then $v^2 = u_1u_2u_3$, where
\[
u_1 = q_1^2 q_2 \cdot \ldots \cdot q_s , \ u_2 = q_1 q_2^2 q_3 \cdot \ldots \cdot q_s, \ \text{and} \ u_3 = q_1q_2q_3^2 q_4^2 \cdot \ldots \cdot q_s^2 \,.
\]
Since $u_1, u_2, u_3 \in \mathcal A (B)$, $v^2$ has the required property. \qed ({\bf A4)}

\smallskip
Thus both {\bf A3} and {\bf A4} hold true. Let ${\ell} \in \N$ and $v \in D_1 \cap \mathcal A (B)$. Then {\bf A3} implies that $2 \in \mathsf L_B (v^{\ell})$ or $2 \in \mathsf L_B (v^{\ell}u)$ for some $u \in \mathcal A (B)$. Thus ${\ell} \in \mathcal U_2 (B)$ or ${\ell}+1 \in \mathcal U_2 (B)$, and this implies that $l+1\in\mathcal U_3 (B)$ or $l+2\in\mathcal U_3 (B)$.

\smallskip
It remains to show that $\mathcal U_k (B)\supset\N_{\geq 4}$ for all $k\geq 4$. By Lemma \ref{2.1}, it suffices to verify that $\mathcal U_4 (B)\supset\N_{\geq 4}$. Let $k\geq 4$. By {\bf A4} there is an element $c\in B$ such that $\{2,3\}\subset\mathsf L_B(c)$. We choose an atom $v\in D_1\cap\mathcal A (B)$. If $2\in\mathsf L_B(v^{k-3})$, then $\{4,k\}\subset\mathsf L_B(cv^{k-3})$. Otherwise there is an $u\in\mathcal A (B)$ such that $2\in\mathsf L_B(uv^{k-3})$, and hence $\{4,k\}\subset\mathsf L_B(u^3v^{k-3})$.
\end{proof}

\begin{remark}\label{R:GenHF}
Let $H$ be a seminormal $v$-noetherian weakly Krull monoid with $\emptyset \ne \mathfrak f = (H \DP
\wh H) \subsetneq H$ such that $H_{\mathfrak p}$ is finitely primary of rank one for each $\mathfrak{p}\in\mathfrak{X}(H)$. Then the proof of 2.(a).(ii) and Corollary \ref{L:CharClsIso} show that, if  $\overline{\vartheta}\colon \mathcal{C}_v(H)\rightarrow \mathcal{C}_v(\wh{H})$ is an isomorphism, then the
inclusion $H \hookrightarrow \wh{H}$ is a transfer homomorphism (without any further assumption on the distribution of prime ideals). We will apply this remark in the next
section.
\end{remark}

\smallskip
We end with a list of examples indicating that the statements of  Theorem \ref{5.6}  need not hold without the assumption on seminormality.

\medskip
\begin{examples} \label{5.12}~

1. There is a one-dimensional  noetherian local domain $(R, \mathfrak m)$ with non-zero conductor $(R \DP \wh R)$ and finite residue field $R/\mathfrak m$ whose monotone catenary degree $\mathsf c_{\monn} (R)$ is infinite  (\cite[Example 6.3]{Ha09c}). Clearly, $R^{\bullet}$ is finitely primary, and thus seminormality is the only assumption of Theorem \ref{5.6} which does not hold here.

\smallskip
2. Let $H$ be a $v$-noetherian primary monoid such that $\wh H$ is Krull. If $(H \DP \wh H) \ne \emptyset$ or $|\mathfrak X (\wh H)| \ge 2$, then $H$ has finite catenary degree. However, there are examples of such monoids $H$ with $(H \DP \wh H) = \emptyset$, $|\mathfrak X (\wh H)| =1$, $\wh H$ a discrete valuation monoid, and $H$ has infinite set of distances $\Delta (H)$ and hence infinite catenary degree $\mathsf c (H)$ (\cite[Theorem 3.5 and Proposition 3.7]{Ge-Ha-Le07}).

\smallskip
3. Suppose that  $k \ge 4$. Then any of the statements in Theorem \ref{5.6}, that $\mathcal U_k (H)$ is a finite non-empty interval or that $\N_{\ge 4} \subset \mathcal U_k (H)$, implies that $\min \Delta (H) = 1$. On the other hand, for every $d \in \N$ there is a   $v$-noetherian finitely primary monoid $S$ with $\min \Delta (S) = d$ (\cite[Example 3.1.9, Corollary 2.9.8]{Ge-HK06a}).

\smallskip
4. It is easy to construct  Krull monoids (or Dedekind domains) with finite class group whose unions of sets of lengths are finite but not arithmetical progressions (\cite[Example 3.1]{Fr-Ge08}). These monoids are seminormal $v$-noetherian (because they are Krull) but they have classes containing no prime divisors. There are also numerical monoids whose unions of sets of lengths are not arithmetical progressions (\cite[Remarks 6.7]{Bl-Ga-Ge11a}; note that a numerical monoid is a $v$-noetherian finitely primary monoid of rank one). On the other hand, Krull monoids with finite class group as well as
numerical monoids have finite monotone catenary degree.
\end{examples}

\bigskip
\section{On the half-factoriality of seminormal weakly Krull monoids} \label{6}
\bigskip

Half-factoriality has been a central topic in factorization theory since its very beginning; see, for example,  \cite{Ch-Co00, Co03, Co05a, Sc05c, Ma-Ok09a, Ro11a, Co-Sm11a, Ma-Ok16a}. In 1960 Carlitz showed that a ring of integers in a number field is half-factorial if and only if the class group has at most two elements. In 1983 Halter-Koch gave a  characterization of half-factorial orders in quadratic number fields (\cite{HK83a}, \cite[Theorem 3.7.15]{Ge-HK06a}). In spite of recent partial results \cite{Pi00, Pi02a, Ka05b, Ph12b}, there is still a long way to go to a characterization of half-factorial orders in a general algebraic number field. Much emphasis has been put on the question of which conditions guarantee that localizations (or more generally, overrings) of half-factorial domains are again half-factorial. A standing conjecture states that overrings of half-factorial orders in algebraic number fields, which are contained in the principal order, are all still half-factorial (\cite[page 323]{Ka05b}).

As in Section \ref{5}, we work in the setting of $v$-noetherian weakly Krull monoids with non-trivial conductor, and we always assume seminormality. Theorem \ref{6.2} offers a  characterization of half-factoriality in natural algebraic terms. Furthermore, we will show that half-factoriality is equivalent to the fact that the catenary degree is at most two. This equivalence holds true for arbitrary Krull monoids provided that every class contains a prime divisor, and it fails without this additional condition on primes. Furthermore, it fails for $v$-noetherian finitely primary monoids which are not seminormal (\cite[Example 3.9]{Ph12b}). As it was in Theorem \ref{5.6}, the assumption that every class contains a prime divisor is crucial. Without this assumption, half-factoriality does not imply that the class group has at most two elements, not even in the Krull monoid setting (see \cite{Sc05c, Ge-Go03}).  Proposition \ref{6.4}  reveals that -- under very mild assumptions  -- overrings of half-factorial domains are half-factorial again. In particular, this confirms the  above mentioned conjecture  in the seminormal case.

\medskip
\begin{lemma} \label{6.1}
Let $H$ be a $v$-noetherian weakly Krull monoid such that $H_{\mathfrak p}$ is finitely primary for each $\mathfrak p \in \mathfrak X (H)$.
Let $\pi \colon \mathfrak X(\wh H) \to \mathfrak X(H)$ be the natural map defined by \,$\pi(\mathfrak P) = \mathfrak P \cap H$ for all $\mathfrak P \in \mathfrak X(\wh H)$.
\begin{enumerate}
\item  Suppose that  $\mathcal C (H)$ is finite.
       \begin{enumerate}
       \item If $H$ is half-factorial, then  $\pi$ is bijective.

       \smallskip
       \item If $H$ is seminormal and half-factorial, then $\mathsf c \big( \mathcal I_v^* (H) \big) \le  2$.
       \end{enumerate}

\smallskip
\item Suppose that $H$ is seminormal and has trivial class group $($i.e., $H$ is seminormal weakly factorial$)$. Then the following statements are equivalent{\rm \,:}
       \begin{enumerate}
       \item $\mathsf c (H) \le 2$.

       \item $H$ is half-factorial.

       \item  $\pi$ is bijective.
       \end{enumerate}
\end{enumerate}
\end{lemma}

\begin{proof}
By Proposition \ref{5.3}, there is an isomorphism $\delta_H \colon \mathcal I_v^*(H) \to D= \coprod_{\mathfrak p \in \mathfrak X(H)} (H_\mathfrak p)_\red$.

1.(a) Recall that  $\partial_H \colon H_\red \to \mathcal I_v^*(H)$, defined by $\partial(aH^\times) = aH$ for all $a \in H$, is a monomorphism, and set $S = \delta_H \big( \partial_H (H_\red) \big) \subset D$. Suppose that $H$ is half-factorial. Then $S$ is half-factorial, and we assume to the contrary that $\pi$ is not bijective. Then, by Lemma \ref{5.1}.3, there is some $\mathfrak p \in \mathfrak X (H)$ such that $H_{\mathfrak p}$ has rank $s \ge 2$, and hence $(H_{\mathfrak p})_{\red} \subset D$ has rank $s$. We suppose it has   exponent $\alpha$ and that $\mathcal C (H)$ has exponent $n$. By Lemma \ref{3.0}, there are,  for every $N \in \N$,  atoms $u_1, \ldots, u_k, v_1, \ldots, v_l \in (H_{\mathfrak p})_{\red}$ such that $u_1 \cdot \ldots \cdot u_k = v_1 \cdot \ldots \cdot v_l$ with $k \le 2 \alpha$ and $l \ge N$.  Since for every atom $w \in (H_{\mathfrak p})_{\red}$, $w^n \in S$ with $\mathsf L_S (w^n) \subset [1,n]$, the equation $u_1^n \cdot \ldots \cdot u_k^n = v_1^n \cdot \ldots \cdot v_l^n$ shows that $S$ is not half-factorial, a contradiction.

\smallskip
1.(b) The assertion follows follows from 1.(a), from Lemma \ref{3.5}.2, and from Theorem \ref{3.7}.

\smallskip
2. (a)\, $\Rightarrow$\, (b) \ follows from the basic inequality (\ref{basic3}), and \
(b)\, $\Rightarrow$\, (c) \ follows from 1.

(c)\, $\Rightarrow$\, (a) \ Suppose that $\pi$ is bijective. Then each  $H_{\mathfrak p}$ has rank one by Lemma \ref{5.1}.3. Thus  Lemma \ref{3.5} and  Theorem \ref{3.7} imply $\mathsf c (H) \le 2$.
\end{proof}

\medskip
We need a further arithmetical invariant.
Let $H$ be an atomic monoid. For $u \in \mathcal A (H)$, we denote by $\omega (H, u)$ the smallest $N \in \N_0 \cup \{\infty\}$ with the following property: if $n \in \N$ and $a_1, \ldots, a_n \in H$ are such that $u$ divides $a_1 \cdot \ldots \cdot a_n$, then $u$ already divides a subproduct consisting of at most $N$ factors. We set
\[
\omega (H) = \sup \{ \omega (H, u) \mid u \in \mathcal A (H) \} \in \N_0 \cup \{\infty\} \,.
\]
By definition, an element $u \in H$ is a prime element of $H$ if and only if $\omega (H, u ) = 1$, and thus $H$ is factorial if and only if $\omega (H) = 1$. If $H$ is not factorial, then $\mathsf c (H) \le \omega (H)$ (\cite[Proposition 3.6]{Ge-Ka10a}).  If $H$ is $v$-noetherian, then $\omega (H, u) < \infty$ for all $u \in \mathcal A (H)$ (\cite[Theorem 4.2]{Ge-Ha08a}).

\medskip
\begin{theorem}[\bf Characterization of half-factoriality] \label{6.2}
Let $H$ be a seminormal $v$-noetherian weakly Krull monoid with $\emptyset \ne \mathfrak f = (H \DP \wh H) \subsetneq H$ such that
$H_{\mathfrak p}$ is finitely primary of rank $s_{\mathfrak p}$ for each $\mathfrak p \in \mathfrak X (H)$.
Let $\pi \colon \mathfrak X(\wh H) \to \mathfrak X(H)$ be the natural map defined by \,$\pi(\mathfrak P) = \mathfrak P \cap H$ for all $\mathfrak P \in \mathfrak X(\wh H)$, and let \,$\overline \vartheta \colon \mathcal C_v(H) \to \mathcal C_v(\wh H)$ be the epimorphism given by {\rm Proposition \ref{exactsequence}}.
Suppose that the class group $G = \mathcal C (H)$ is finite, and that every class contains a $\mathfrak p \in \mathfrak X (H)$ with $\mathfrak p \not\supset \mathfrak f$.
\begin{enumerate}
\item The following statements are equivalent{\rm \,:}
      \begin{enumerate}
      \smallskip
      \item $\mathsf c (H) \le 2$.

      \smallskip
      \item $H$ is half-factorial.

      \smallskip
      \item $|G| \le 2$, and the maps $\pi$ and $\overline{\vartheta}$ are both bijective.
      \end{enumerate}

\smallskip
\item Suppose that $|G|=2$ and that $H$ is not half-factorial.
      \begin{enumerate}
      \item Suppose that  $\pi$ is bijective. If there is precisely one $\mathfrak p \in \mathfrak X (H)$ such that $|\{ [u] \in G \mid u \in \mathcal A (H_{\mathfrak p})\}| > 1$,  then $\mathsf c (H) = \omega (H) = 3$ and $\Delta (H) = \{1\}$. Otherwise, we have $\mathsf c (H) = \omega (H) = 4$ and $\Delta (H) = [1,2]$.

      \smallskip
      \item If  $\pi$ is not bijective, then $\omega (H) = \infty$.
      \end{enumerate}
\end{enumerate}
\end{theorem}

\begin{proof}
We set $\mathcal P^* = \{\mathfrak p \in \mathfrak X (H) \mid \mathfrak p \supset \mathfrak f \}$ and $\mathcal P = \mathfrak X (H) \setminus \mathcal P^*$.
Arguing as at the beginning of the proof of Theorem \ref{5.6}, we may suppose that
\[
H \subset D = \mathcal F (\mathcal P) \times \prod_{i=1}^n D_i
\]
is saturated with finite class group $G$, and that every class contains a prime $p \in \mathcal P$. Moreover, for each $i \in [1,n]$, $D_i$ is a reduced seminormal finitely primary monoid of rank $s_i$, and whenever $s_i=1$, we suppose that $D_i \subset \wh{D_i} = \wh{D_i}^{\times} \time [q_i]$.
We set $T = \prod_{i=1}^n D_i$ and $B = \mathcal B (G, T, \iota) \subset F = \mathcal F (G) \time T$.
By Lemma \ref{4.2}, $H$ is half-factorial if and only if $B$ is half-factorial.

\smallskip
1. (a)\, $\Rightarrow$\, (b) \ This follows from the basic inequality (\ref{basic3}).

(b)\, $\Rightarrow$\, (c) \ If $|G| \ge 3$, then $\mathcal B (G)$ is not half-factorial, and since $\mathcal B (G) \subset B$ is divisor-closed, $B$ is not  half-factorial, a contradiction. Thus $|G| \le 2$. Lemma \ref{6.1} implies that  $\pi$ is bijective. Assume to the contrary that $\overline \vartheta$ is not an isomorphism. Then Corollary \ref{L:CharClsIso} implies that $|\{ [u] \in G \mid u \in \mathcal A (D_i)\}| > 1$ for some $i \in
[1,n]$. In particular we obtain that $|G|=2$, say $G = \{0, g\}$. We show that $B$ is not half-factorial, a contradiction to assumption (b).

By Lemma \ref{3.5}, we get
$\mathcal A (D_i) = \{q_i \epsilon \mid \epsilon \in \wh{D_i}^{\times}\}$, and
after replacing $q_i$ by some $q_i \epsilon$ if necessary we may assume that $[q_i] = 0 \in G$, and we pick some $\epsilon \in \wh{D_i}^{\times}$
such that $[q_i \epsilon] = g$. If $\epsilon \in D_i$, then $\epsilon \in D_i  \cap \wh{D_i}^{\times} = D_i^{\times} = \{1\}$, a contradiction. Thus $\epsilon \notin T$. Then $a = g^2 q_i^2 \in B$, and we assert that $\mathsf L_B (a) = \{2,3\}$, which obviously implies that
$\min \Delta (B) = 1$ and hence $B$ is not half-factorial. Since $q_i \in \mathcal A (B)$ and $g^2 \in \mathcal A (B)$, it follows that $3 \in \mathsf L_B (a)$. Since $[g \epsilon q_i] = 0$,
we infer that $g \epsilon q_i \in B$. Since $\epsilon g \notin \mathcal F (G) \time T \supset B$, we get that  $g \epsilon q_i$ is an atom in $B$. Since all this holds true also for $g \epsilon^{-1} q_i$,
it follows that $2 \in \mathsf L_B (a)$. Since $\max \mathsf L_B (a) < 4$, we get that $\mathsf L_B (a) = \{2,3\}$.

\smallskip
(c)\, $\Rightarrow$\, (a) \ By Theorem \ref{5.6}.2.(a), there exists a
transfer homomorphism $\theta \colon B \to \mathcal B (G)$. Since
$|G| \le 2$, $\mathcal B (G)$ is factorial, and hence $\mathsf  c \big( \mathcal B(G) \big) = 0$. Thus Lemma \ref{4.1} and Lemma \ref{4.2} imply that
\[
\mathsf c (H) \le \max \big\{\mathsf c \big( \mathcal B(G) \big), \mathsf c (H, \boldsymbol \beta) \big\} = 2 \,.
\]

\smallskip
2. We set $G = \{0,g\}$.

2.(a) Since, for each $k \in \N$,  the $\mathcal U_k (H)$ are a finite intervals by Theorem \ref{5.6}.2 and since they are not all equal to $\{k\}$ by assumption,  it follows that $1 \in \Delta (H)$. We show that $\omega (H, u) \le 4$ for all $u \in \mathcal A (H)$, and in the case when the upper bound $4$ is attained, that  $2 \in \Delta (H)$.
Then all assertions follow from
\begin{equation}
2 + \max \Delta (H) \le \mathsf c (H) \le \omega (H) \,,  \label{basic4}
\end{equation}
where the last inequality stems from \cite[Proposition 3.6]{Ge-Ka10a}.
Let $u \in \mathcal A (H)$. If $u$ is a prime, then $\omega (H, u) = 1$.  The non-prime atoms  $u \in H$ have one of the following four forms:
\begin{itemize}
\item $u = p_1p_2$, where $p_1, p_2 \in \mathcal P$ with $[p_1]=[p_2] = g$.

\item $u = p v$, where $p \in \mathcal P$ and $v \in \mathcal A (D_1) \cup \ldots \cup \mathcal A (D_n)$ with $[p]=[v]=g$.

\item $u \in H \cap \big( \mathcal A (D_1) \cup \ldots \cup \mathcal A (D_n) \big)$.

\item $u = v_1 v_2$, where $v_1, v_2 \in \mathcal A (D_1) \cup \ldots \cup \mathcal A (D_n)$ with $[v_1]=[v_2]=g$.
\end{itemize}
It is easy to see that $\omega (H, u) \le 3$ if $u$ has one of the first three forms. Suppose that $u = v_1v_2$ with $v_1, v_2$ as above.
After renumbering if necessary we may suppose that $v_1 \in D_1$ and $v_2 \in D_2$. If there is an $i \in [1,2]$ with $|\{ [w] \in G \mid w \in \mathcal A (D_1)\}| = 1$, then $\omega (H, u) \le  3$. Now suppose that $|\{ [w] \in G \mid u \in \mathcal A (D_i)\}| > 1$ for $i \in [1,2]$. Then there are atoms $w_i = \epsilon_i q_i \in H \cap \mathcal A (D_i)$ for $i \in [1,2]$. Clearly, $u$ divides $w_1^2 w_2^2$ in $D$ and hence in $H$, but it does not divide a subproduct in $H$. Thus $\omega (H, u) \ge  4$, and it is easy to see that $\omega (H, u) \le 4$. Since $\mathsf L_H (w_1^2 w_2^2) = \{2,4\}$, it follows that $2 \in \Delta (H)$.

\smallskip
2.(b) Suppose that  $\pi$ is not bijective. Since $\mathcal U_2 (H)$ is infinite by Theorem \ref{5.6}.2, it follows that $\omega (H) = \infty$ by \cite[Proposition 3.6]{Ge-Ka10a}.
\end{proof}

\smallskip
There are monoids $H$ and $D$ such that $D$ is an \FF-monoid with $\mathsf c (D) = 3$, $H \subset D$ is  saturated and cofinal, the class group $\mathcal C (H, D)$ has two elements but $\mathsf c (H) = \infty$ (see \cite[Example 3.6.2]{Ge-HK06a}). This shows that the results on the catenary degree in Theorem \ref{6.2} heavily depend on the very special setting of Theorem \ref{6.2}.

In the sequel a  domain $R$ is called primary if the monoid $R^{\bullet}$ is primary (equivalently, if $R$ is local and one-dimensional; see Lemma \ref{3.3}.2).

\medskip
\begin{proposition} \label{6.3}
Let $R$ be a seminormal one-dimensional Mori domain such that $\wh{R_{\mathfrak p}}$ is a discrete valuation domain for each ${\mathfrak p}\in\max(R)$. Let $K$ be a quotient field of $R$ and let $R\subset S \subsetneq K$ be an intermediate ring.
\begin{enumerate}
\item $\wh{R}$ and $\wh{S}$ are Dedekind domains and $\wh{S_{\mathfrak q}}=\wh{R_{{\mathfrak q}\cap R}}$ is a discrete valuation domain for all ${\mathfrak q}\in {\rm{spec}}(S)^{\bullet}$.

\smallskip
\item For all ${\mathfrak q}\in {\rm{spec}}(S)^{\bullet}$ such that $S_{\mathfrak q}$ is archimedean,  $S_{\mathfrak q}$ is a seminormal primary Mori domain.

\smallskip
\item $\{{\mathfrak m}\cap S\mid {\mathfrak m}\in\max(\wh{S})\}=\mathfrak{X}(S)$ and $\{{\mathfrak q}\in\mathfrak{X}(S)\mid x\in {\mathfrak q}\}$ is finite for all $x\in S^{\bullet}$.

\smallskip
\item If $(R \DP \wh{R})\not=\{0\}$, then $(S \DP\wh{S})\not=\{0\}$.

\smallskip
\item The following statements are equivalent{\rm \,:}
      \begin{enumerate}
      \item $\bigcap_{{\mathfrak q}\in\mathfrak{X}(S)} S_{\mathfrak q}=S$.
      \item $S_{\mathfrak q}$ is archimedean for all ${\mathfrak q}\in\max(S)$.
      \item $S$ is a seminormal one-dimensional Mori domain.
      \end{enumerate}

\smallskip
\item Suppose that the equivalent statements in $(5.)$ are satisfied.
      Further, let ${\mathfrak q}\in\max(S)$ and ${\mathfrak p}\in\max(R)$ be such that ${\mathfrak q}\cap R={\mathfrak p}$. Suppose for all $u,v\in\mathcal{A}(R_{\mathfrak p})$, there exists some $c\in K^{\bullet}$ such that $uR_{\mathfrak p}\cap R=c(vR_{\mathfrak p}\cap R)$. Then for all $u,v\in\mathcal{A}(S_{\mathfrak q})$ there is some $c\in K^{\bullet}$ such that $uS_{\mathfrak q}\cap S=c(vS_{\mathfrak q}\cap S)$.
\end{enumerate}
\end{proposition}

\begin{proof}
1. Since $R$ is a seminormal Mori domain we obtain that $\wh{R}$ is a Krull domain by \cite[Theorem 2.9]{Ba94}. Since $R$ is one-dimensional it follows by \cite[Proposition 2.10.5.1(c)]{Ge-HK06a} that $\wh{R}$ is one-dimensional, hence $\wh{R}$ is a Dedekind domain. Moreover, $\wh{S}$ is an overring of $\wh{R}$, hence $\wh{S}$ is a Dedekind domain. Let ${\mathfrak q}\in {\rm{spec}}(S)^{\bullet}$ and ${\mathfrak p}={\mathfrak q}\cap R$. Then ${\mathfrak p}\in\max(R)$. Furthermore, $R_{\mathfrak p}\subset S_{\mathfrak q}$, and thus $\wh{R_{\mathfrak p}}\subset\wh{S_{\mathfrak q}}$. Since $R_{\mathfrak p}$ is primary, there is some $x\in R_{\mathfrak p}^{\bullet}$ such that $K=R_{\mathfrak p}[x^{-1}]$, hence $K=S_{\mathfrak q}[x^{-1}]$. Assume that $\wh{S_{\mathfrak q}}=K$. Then there is some $c\in S_{\mathfrak q}^{\bullet}$ such that $cK=cS_{\mathfrak q}[x^{-1}]\subset S_{\mathfrak q}$. Consequently, $y=c\frac{y}{c}\in cK\subset S_{\mathfrak q}$ for all $y\in K$, and thus $S_{\mathfrak q}=K$, a contradiction. This implies that $\wh{R_{\mathfrak p}}\subset\wh{S_{\mathfrak q}}\subsetneq K$ and since $\wh{R_{\mathfrak p}}$ is a discrete valuation domain it follows that $\wh{S_{\mathfrak q}}=\wh{R_{\mathfrak p}}$ is a discrete valuation domain.

\smallskip
2. Let ${\mathfrak q}\in {\rm{spec}}(S)^{\bullet}$ be such that $S_{\mathfrak q}$ is archimedean and ${\mathfrak p}={\mathfrak q}\cap R$. It follows by 1. that $R_{\mathfrak p}\subset S_{\mathfrak q}\subset\wh{S_{\mathfrak q}} = \wh{R_{\mathfrak p}}$. Therefore, Lemma \ref{3.2}.4 implies that $S_{\mathfrak q}$ is a seminormal primary Mori domain.

\smallskip
3. $\subset$: Let ${\mathfrak m}\in\max(\wh{S})$ and ${\mathfrak q}={\mathfrak m}\cap S$. Observe that $S_{\mathfrak q}\subset\wh{S}_{\mathfrak m}$, and ${\mathfrak m}_{\mathfrak m}\cap S_{\mathfrak q}\subset {\mathfrak q}_{\mathfrak q}$ and ${\mathfrak q}_{\mathfrak q}=(S\setminus {\mathfrak q})^{-1}{\mathfrak q}\subset (\wh{S}\setminus {\mathfrak m})^{-1}{\mathfrak m}={\mathfrak m}_{\mathfrak m}$. Therefore, ${\mathfrak q}_{\mathfrak q}={\mathfrak m}_{\mathfrak m}\cap S_{\mathfrak q}$, hence $S_{\mathfrak q}\cap\wh{S}_{\mathfrak m}^{\times}=S_{\mathfrak q}^{\times}$. Since $\wh{S_{\mathfrak q}}\subset\wh{S}_{\mathfrak m}\subsetneq K$ and $\wh{S_{\mathfrak q}}$ is a discrete valuation domain, we have $\wh{S_{\mathfrak q}}=\wh{S}_{\mathfrak m}$, and thus $S_{\mathfrak q}\cap\wh{S_{\mathfrak q}}^{\times}=S_{\mathfrak q}\cap\wh{S}_{\mathfrak m}^{\times}=S_{\mathfrak q}^{\times}$, which implies that $S_{\mathfrak q}$ is archimedean. Consequently, $S_{\mathfrak q}$ is primary by 2., hence ${\mathfrak q}\in\mathfrak{X}(S)$.

$\supset$: Let ${\mathfrak q}\in\mathfrak{X}(S)$ and ${\mathfrak m}=\wh{S_{\mathfrak q}}\setminus\wh{S_{\mathfrak q}}^{\times}\cap\wh{S}$. Then ${\mathfrak m}\in\max(\wh{S})$. Since $S_{\mathfrak q}$ is primary, it is archimedean and hence it follows by 2. that $S_{\mathfrak q}$ is seminormal. Therefore, $(S_{\mathfrak q} \DP \wh{S_{\mathfrak q}})$ is a radical ideal of $S_{\mathfrak q}$ and of $\wh{S_{\mathfrak q}}$ and $(S_{\mathfrak q} \DP \wh{S_{\mathfrak q}})\not=\{0\}$ by Lemma \ref{3.2}.1. Since $S_{\mathfrak q}$ and $\wh{S_{\mathfrak q}}$ are primary this implies that ${\mathfrak q}_{\mathfrak q}=S_{\mathfrak q}\setminus S_{\mathfrak q}^{\times}=(S_{\mathfrak q} \DP \wh{S_{\mathfrak q}})=\wh{S_{\mathfrak q}}\setminus\wh{S_{\mathfrak q}}^{\times}$. Consequently, ${\mathfrak m}\cap S={\mathfrak q}_{\mathfrak q}\cap S={\mathfrak q}$.\\
Let $x\in S^{\bullet}$. It follows that $\varphi \colon \{{\mathfrak m}\in\max(\wh{S})\mid x\in {\mathfrak m}\}\rightarrow\{{\mathfrak q}\in\mathfrak{X}(S)\mid x\in {\mathfrak q}\}$ defined by $\varphi({\mathfrak m})={\mathfrak m}\cap S$ is a surjective map. Since $\wh{S}$ is a Dedekind domain we have $\{{\mathfrak m}\in\max(\wh{S})\mid x\in {\mathfrak m}\}$ is finite, and thus $\{{\mathfrak q}\in\mathfrak{X}(S)\mid x\in {\mathfrak q}\}$ is finite.

\smallskip
4. Let $(R \DP \wh{R})\not=\{0\}$. First we show that $(R \DP \wh{R})\subset (S_{\mathfrak q} \DP \wh{S}_{\mathfrak q})$ for all ${\mathfrak q}\in\max(S)$. Let ${\mathfrak q}\in\max(S)$ and ${\mathfrak p}={\mathfrak q}\cap R$. Since $R$ is a Mori domain and $\wh{S}_{\mathfrak q}\subset\wh{S_{\mathfrak q}}$ it follows by 1. that $(R \DP \wh{R})\subset (R_{\mathfrak p} \DP \wh{R_{\mathfrak p}})\subset (S_{\mathfrak q} \DP \wh{S_{\mathfrak q}})\subset (S_{\mathfrak q} \DP \wh{S}_{\mathfrak q})$. Let $x\in\bigcap_{{\mathfrak m}\in\max(S)} (S_{\mathfrak m} \DP \wh{S}_{\mathfrak m})$. Then $x\wh{S}_{\mathfrak m}\subset S_{\mathfrak m}$ for all ${\mathfrak m}\in\max(S)$. Therefore, $x\wh{S}=x\bigcap_{{\mathfrak m}\in\max(S)}\wh{S}_{\mathfrak m}=\bigcap_{{\mathfrak m}\in\max(S)} x\wh{S}_{\mathfrak m}\subset\bigcap_{{\mathfrak m}\in\max(S)} S_{\mathfrak m}=S$. This implies that $x\in (S \DP \wh{S})$. Consequently, $(R \DP \wh{R})\subset\bigcap_{{\mathfrak m}\in\max(S)} (S_{\mathfrak m} \DP \wh{S}_{\mathfrak m})\subset (S \DP \wh{S})$, and thus $(S \DP \wh{S})\not=\{0\}$.

\smallskip
5. We use freely that $R$ is primary if and only if $R$ is one-dimensional and local (Lemma \ref{3.3}.2.

(a)\, $\Rightarrow$\, (b) It follows by 2. that $S_{\mathfrak q}$ is a Mori domain for all ${\mathfrak q}\in\mathfrak{X}(S)$. By 3., we infer that $S$ is a weakly Krull domain and hence Mori by Lemma \ref{5.1}.1. Consequently, $S_{\mathfrak m}$ is a Mori domain for all ${\mathfrak m}\in\max(S)$. This implies that $S_{\mathfrak m}$ is archimedean for all ${\mathfrak m}\in\max(S)$.

(b)\, $\Rightarrow$\, (c) It follows by 2. that $S_{\mathfrak q}$ is a seminormal primary Mori domain for all ${\mathfrak q}\in\max(S)$. It follows immediately that $S$ is seminormal and one-dimensional. Together with 3. we have $S$ is a Mori domain.

(c)\, $\Rightarrow$\, (a) Trivial, since $S$ is one-dimensional.

\smallskip
6. By Lemma \ref{3.3}.3, $R_{\mathfrak p}^{\bullet}$ is seminormal and finitely primary, and by 1. we have $R_{\mathfrak p}\subset S_{\mathfrak q}\subset\wh{R_{\mathfrak p}}$. Consequently, $\mathcal{A}(R_{\mathfrak p})=\mathcal{A}(S_{\mathfrak q})$ by Lemma \ref{3.5}.4. Next we show that $(wR_{\mathfrak p}\cap R)S=wS_{\mathfrak q}\cap S$ for all $w\in\mathcal{A}(R_{\mathfrak p})$. Let $w\in\mathcal{A}(R_{\mathfrak p})$. It suffices to prove that $(wR_{\mathfrak p}\cap R)S_{\mathfrak n}=(wS_{\mathfrak q}\cap S)S_{\mathfrak n}$ for all ${\mathfrak n}\in\max(S)$. Let ${\mathfrak n}\in\max(S)$. Observe that $wR_{\mathfrak p}\cap R$ is ${\mathfrak p}$-primary and $wS_{\mathfrak q}\cap S$ is ${\mathfrak q}$-primary. Therefore, $(wR_{\mathfrak p}\cap R)R_{\mathfrak a}=R_{\mathfrak a}$ for all ${\mathfrak a}\in\max(R)\setminus\{{\mathfrak p}\}$ and $(wS_{\mathfrak q}\cap S)S_{\mathfrak b}=S_{\mathfrak b}$ for all ${\mathfrak b}\in\max(S)\setminus\{{\mathfrak q}\}$.

\smallskip
\noindent CASE 1: \ ${\mathfrak n}\not={\mathfrak q}$.

Assume that ${\mathfrak n}\cap R={\mathfrak p}$. It follows by 1. that $\wh{R_{\mathfrak p}}=\wh{S_{\mathfrak q}}=\wh{S_{\mathfrak n}}$ is a discrete valuation domain. Since $\wh{R_{\mathfrak p}}\setminus\wh{R_{\mathfrak p}}^{\times}\cap S_{\mathfrak q}={\mathfrak q}_{\mathfrak q}$ and $\wh{R_{\mathfrak p}}\setminus\wh{R_{\mathfrak p}}^{\times}\cap S_{\mathfrak n}={\mathfrak n}_{\mathfrak n}$, we obtain that ${\mathfrak q}={\mathfrak q}_{\mathfrak q}\cap S=\wh{R_{\mathfrak p}}\setminus\wh{R_{\mathfrak p}}^{\times}\cap S={\mathfrak n}_{\mathfrak n}\cap S={\mathfrak n}$, a contradiction. Therefore, ${\mathfrak n}\cap R\not={\mathfrak p}$, hence $(wR_{\mathfrak p}\cap R)S_{\mathfrak n}=(wR_{\mathfrak p}\cap R)R_{{\mathfrak n}\cap R}S_{\mathfrak n}=R_{{\mathfrak n}\cap R}S_{\mathfrak n}=S_{\mathfrak n}=(wS_{\mathfrak q}\cap S)S_{\mathfrak n}$.

\smallskip
\noindent CASE 2: \ ${\mathfrak n}={\mathfrak q}$.

We have $(wR_{\mathfrak p}\cap R)S_{\mathfrak q}=(wR_{\mathfrak p}\cap R)R_{\mathfrak p}S_{\mathfrak q}=wR_{\mathfrak p}S_{\mathfrak q}=wS_{\mathfrak q}=(wS_{\mathfrak q}\cap S)S_{\mathfrak q}$.\\
Now let
$u,v\in\mathcal{A}(S_{\mathfrak q})$. Then there is some $c\in K^{\bullet}$ such that $uR_{\mathfrak p}\cap R=c(vR_{\mathfrak p}\cap
R)$, hence $uS_{\mathfrak q}\cap S=(uR_{\mathfrak p}\cap R)S=c(vR_{\mathfrak p}\cap R)S=c(vS_{\mathfrak q}\cap S)$.
\end{proof}

\medskip
\begin{proposition} \label{6.4}
Let $R$ be a  half-factorial seminormal one-dimensional Mori domain with $\{0\} \ne \mathfrak{f}=(R \DP
\wh{R}) \subsetneq R$. Suppose that  $\Pic(R)$ is finite and that every class contains a maximal ideal ${\mathfrak p}$ with ${\mathfrak p}
\not\supset \mathfrak{f}$. Let $K$ be a quotient field of $R$ and $R\subset S\subsetneq K$ an
intermediate ring.
\begin{enumerate}
\item  If the equivalent conditions of Proposition \ref{6.3}.5 are satisfied, then $S$ is
half-factorial.

\smallskip
\item If the integral closure $\overline R$ equals the complete integral closure $\wh R$, then the equivalent conditions of Proposition \ref{6.3}.5 are satisfied and $S$ is
half-factorial.
\end{enumerate}
\end{proposition}

\begin{proof}
Since $R$ is a one-dimensional Mori domain, the $v$-class group $\mathcal C_v (R)$ coincides with the Picard group $\Pic (R)$.

\smallskip
1. By Proposition \ref{6.3}.1, $\wh{R}\subset\wh{S}$ are Dedekind domains. By Theorem \ref{6.2}, the class groups of $R$ and $\wh{R}$ are isomorphic. In particular we conclude,
that the class group of $\wh{R}$ is a torsion group. Hence (see for example \cite[Theorem
40.3]{Gi92}) $\wh{S}=T^{-1}\wh{R}$ for some submonoid $T\subset \wh{R}^\bullet$. By a Theorem of Nagata (\cite[Corollary 7.2]{Fo73}), the class group of $\wh{S}$ is a quotient of the class group of $\wh{R}$. By Theorem \ref{6.2}, we know $| \Pic (\wh{R})| \le 2$, so that
$\vert \Pic (\wh{S}) \vert \leq 2$ holds true, too. Therefore, $\wh{S}$ is half-factorial.

By Proposition \ref{6.3}.6 and Corollary \ref{L:CharClsIso},  the homomorphism
$\Pic (S)\rightarrow \Pic (\wh{S})$ is an isomorphism. By Proposition \ref{6.3} and Remark \ref{R:GenHF},
the inclusion $S^\bullet \hookrightarrow \wh{S}^\bullet$ is a transfer homomorphism, which implies that $S$ is
half-factorial, too.

\smallskip
2. We show that $S_{\mathfrak q}$ is archimedean for all ${\mathfrak q}\in\max(S)$. Let ${\mathfrak q}\in\max(S)$ and ${\mathfrak p}={\mathfrak q}\cap R$. Since $R$ is a Mori domain we obtain that $\overline{R_{\mathfrak p}}=\overline{R}_{\mathfrak p}=\wh{R}_{\mathfrak p}=\wh{R_{\mathfrak p}}$. Consequently, Proposition \ref{6.3}.1 implies that $\wh{R_{\mathfrak p}}=\overline{R_{\mathfrak p}}\subset\overline{S_{\mathfrak q}}\subset\wh{S_{\mathfrak q}}=\wh{R_{\mathfrak p}}$, hence $\wh{S_{\mathfrak q}}=\overline{S_{\mathfrak q}}$. Consequently, $\wh{S_{\mathfrak q}}^{\times}\cap S_{\mathfrak q}=\overline{S_{\mathfrak q}}^{\times}\cap S_{\mathfrak q}=S_{\mathfrak q}^{\times}$, and thus $S_{\mathfrak q}$ is archimedean.
\end{proof}

\medskip
In the following examples we provide non-noetherian examples for rings in Proposition \ref{6.4}, and we outline the necessity of the assumptions there.

\medskip
\begin{examples} \label{6.5}~

1. Let $K$ be a field, $Y$ an indeterminate over $K$ and $X$ an indeterminate over $K(Y)$. Set $R=K+X(K(Y))[\![X]\!]$, $S=K[Y^2,Y^3]+X(K(Y))[\![X]\!]$ and $T=K[Y]+X(K(Y))[\![X]\!]$. Then $R$ is an integrally closed primary Mori domain with $(R \DP \wh{R})\not=\{0\}$ and $\wh{R}$ is a discrete valuation domain. Moreover, $R\subset S\subset T\subset\wh{R}$ are intermediate rings, $S$ is neither seminormal nor archimedean and $T$ is integrally closed and yet not archimedean.

\smallskip
2. Let $K$ be a field, $Y$ an indeterminate over $K$, $L$ a field of quotients of $K[\![Y]\!]$ and $X$ an indeterminate over $L$. Set $R=K[\![Y]\!]+XL[\![X]\!]$. Then $R$ is a non-archimedean integrally closed domain with  $(R \DP \wh{R})\not=\{0\}$, $|{\rm{spec}}(R)|=3$, and $\wh{R}$ is a discrete valuation domain. In particular, $R$ is a seminormal $G$-domain whose complete integral closure is a Krull domain and yet $R$ is not a Mori domain.

\smallskip
3.
Let $K$ be a field, $L/K$ an algebraic field extension such that $[L:K]=\infty$, $X$ an indeterminate over $L$ and $R=K+XL[X]$. Then $R$ is a weakly factorial half-factorial seminormal one-dimensional Mori domain with $\mathfrak{f}=(R \DP \wh{R})\not=\{0\}$, $\overline{R}=\wh{R}$ and yet $R$ is not noetherian. Moreover, $\Pic(R)$ is trivial and every class of $\Pic(R)$ contains a $\mathfrak{p}\in\max(R)$ such that $\mathfrak{f}\not\subset\mathfrak{p}$.
\end{examples}

\medskip
\noindent
{\bf Acknowledgements.} We would like to thank Franz Halter-Koch and the anonymous referee. Both have read a previous version of the manuscript extremely carefully and have provided us with many helpful comments. A special thanks to Franz Halter-Koch for his help with Example \ref{5.7}.(7).

\providecommand{\bysame}{\leavevmode\hbox to3em{\hrulefill}\thinspace}
\providecommand{\MR}{\relax\ifhmode\unskip\space\fi MR }
\providecommand{\MRhref}[2]{%
  \href{http://www.ams.org/mathscinet-getitem?mr=#1}{#2}
}
\providecommand{\href}[2]{#2}

\end{document}